\documentclass[10pt]{amsart}

\numberwithin{equation}{section}
\newcounter{mnote}
\let\oldmarginpar\marginpar
\renewcommand\marginpar[1]{\-\oldmarginpar[\raggedleft\footnotesize #1]%
{\raggedright\footnotesize #1}}

\usepackage[english]{babel}

\usepackage{amssymb}
\usepackage{mathrsfs}
\usepackage{stmaryrd}
\usepackage{chemarrow}
\usepackage[norelsize]{algorithm2e}
\usepackage{enumerate}
\usepackage{graphicx}
\usepackage{float}
\usepackage{subcaption}
\usepackage[all]{xy}
\usepackage{tikz}
\usetikzlibrary{arrows}
\usepackage{multirow}
\usepackage[colorinlistoftodos]{todonotes}
\usepackage[colorlinks=true, allcolors=blue]{hyperref}
\usepackage[hyperpageref]{backref} 
\usepackage{hyperref}
\hypersetup{hypertex=true,
colorlinks=true,
linkcolor=blue,
anchorcolor=blue,
citecolor=red}
\usepackage{verbatim}
\usepackage{booktabs}
\allowdisplaybreaks

\newtheorem{theorem}{Theorem}[section]
\newtheorem{lemma}[theorem]{Lemma}
\newtheorem{corollary}[theorem]{Corollary}

\newtheorem{example}[theorem]{Example}

\newtheorem{remark}[theorem]{Remark}
\newtheorem{assumption}[theorem]{Assumption}

\newcommand{\normmm}[1]{{\left\vert\kern-0.25ex\left\vert\kern-0.25ex\left\vert #1
\right\vert\kern-0.25ex\right\vert\kern-0.25ex\right\vert}}
\newcommand{\dx}{\,{\rm d}x}
\newcommand{\ds}{\,{\rm d}s}

\newcommand{\curl}{\operatorname{curl}}
\renewcommand{\div}{\operatorname{div}}
\newcommand{\grad}{\operatorname{grad}}

\newcommand{\rot}{\operatorname{rot}}

\newcommand{\skw}{\operatorname{skw}}

\newcommand{\Oplus}{\ensuremath{\vcenter{\hbox{\scalebox{1.5}{$\oplus$}}}}}

\begin{document}

\title[IP nonconforming FEM for SGE]{An interior penalty nonconforming finite element method for strain gradient elasticity}

\author{Xuehai Huang}%
\address{School of Mathematics, Shanghai University of Finance and Economics, Shanghai 200433, China}%
\email{huang.xuehai@sufe.edu.cn}%
\author{Zheqian Tang}%
\address{School of Mathematics, Shanghai University of Finance and Economics, Shanghai 200433, China}%
\email{tangzq0329@163.com}%

\thanks{The first author was supported by the National Natural Science Foundation of China (Grant No.\ 12671432).}


\begin{abstract}
An interior penalty nonconforming finite element method is developed for a strain gradient elasticity (SGE) model in arbitrary dimensions, using an $H^1$-nonconforming displacement element built from vector-valued quadratic polynomials enriched by divergence-free face bubbles. Rigorous analysis establishes optimal and robust error estimates with respect to both the Lam\'{e} coefficient $\lambda$ and the size parameter $\iota$. The underlying divergence-commuting structure further leads to nonconforming finite element Stokes complexes in two and three dimensions. Numerical experiments support the predicted convergence behavior and parameter robustness, while a circular-hole benchmark under uniaxial tension illustrates the applicability of the method to a multiply connected domain.
\end{abstract}



\maketitle

\section{Introduction}

Many microstructured solids exhibit size-dependent responses when the characteristic scale of deformation becomes comparable to an intrinsic material length, a phenomenon that classical elasticity cannot describe. Strain gradient theories account for such effects by incorporating higher-order strain measures. The Toupin--Mindlin theory provides a general framework with several additional material parameters \cite{toupin1962elastic,mindlin1964micro,mindlin1968first}, while the simplified strain gradient elasticity (SGE) model of Aifantis et al. \cite{altan1992structure,ru1993simple} involves only one size parameter and has been widely used in engineering applications \cite{askes2011gradient,polizzotto2015unifying}. In this paper, we consider this reduced model on a bounded polytopal domain $\Omega\subset\mathbb{R}^d$ ($d\geq 2$):
\begin{equation}\label{SGE0}
\begin{cases}
-\div ((\boldsymbol{I}-\iota^{2}\Delta)\boldsymbol{\sigma}(\boldsymbol{u}))=\boldsymbol{f} &\mbox{in} \ \Omega,\\
\boldsymbol{u}=\partial_{n}\boldsymbol{u}=\boldsymbol{0} &\mbox{on} \ \partial\Omega,
\end{cases}
\end{equation}
where $\boldsymbol{u}:\Omega\to\mathbb R^d$ is the displacement and $\boldsymbol{f}\in L^2(\Omega;\mathbb R^d)$ is the body force. The linearized strain and Cauchy stress are, respectively,
\[
\boldsymbol{\varepsilon}(\boldsymbol{u})=\frac{1}{2}\big(\nabla\boldsymbol{u}+(\nabla\boldsymbol{u})^\intercal\big),
\qquad
\boldsymbol{\sigma}(\boldsymbol{u})=2\mu\boldsymbol{\varepsilon}(\boldsymbol{u})+\lambda(\div\boldsymbol{u})\boldsymbol{I},
\]
with Lam\'{e} coefficients $\mu>0$ and $\lambda\geq\lambda_0>0$. Here $\boldsymbol{I}$ is the identity tensor, $\partial_n\boldsymbol{u}$ is the normal derivative of $\boldsymbol{u}$ on $\partial\Omega$, and $\iota\in(0,1]$ is the size parameter. Problem~\eqref{SGE0} is a fourth-order singular perturbation of linear elasticity. As $\iota\to0$, its limiting solution generally fails to satisfy the higher-order boundary condition $\partial_n\boldsymbol{u}=\boldsymbol{0}$, and a boundary layer develops. Independently, large values of $\lambda$ drive the material toward incompressibility and may cause locking. The fourth-order structure and these two parameter limits make it challenging to construct a low-order method that is both optimal and robust with respect to $\iota$ and $\lambda$.

Existing finite element methods address the fourth-order structure in
several ways. The primal displacement formulation has been discretized
using $H^2$-conforming elements
\cite{zervos2001modelling,zervos2009two} and various
$H^2$-nonconforming elements
\cite{MR3712289,MR3912981,MR4296093}. Parameter-robust methods with
respect to both $\iota$ and $\lambda$ were developed in
\cite{ChenHuangHuang2023,MR4549866,tianshudan}. Their uniform
energy-norm estimate has the sharp, but suboptimal, convergence rate
$\mathcal O(h^{1/2})$. 
The loss of order was subsequently overcome in
\cite{ChenHuangHuang2025,HuangHuangTang2024} by imposing the higher-order boundary condition
through Nitsche's technique, leading to optimal and
parameter-robust convergence. The displacement spaces employed in
these two methods are nonconforming discretizations of the short
complex
\[
H_0^2(\Omega;\mathbb R^d)
\xrightarrow{\div}
H_0^1(\Omega)\cap L_0^2(\Omega)
\longrightarrow0.
\]

Mixed formulations provide another route by reducing the regularity
required of the displacement at the cost of introducing additional
stress-type variables; see, for example, \cite{HuangTang2026,chirkov2024mixed-2,riesselmann2024efficient}.
Interior penalty methods instead relax interelement continuity while
retaining a primal formulation. They have long been used for
fourth-order elliptic problems; see the early works \cite{douglasdupont1976,Baker1977,wheeler1978} and
the $C^0$ interior penalty analysis in \cite{BrennerSungIP2005}. Applications to
strain-gradient models include the one-dimensional Toupin--Mindlin
formulation in \cite{EngelEtAl2002} and the two- and three-dimensional
flexoelectricity method in \cite{VenturaCodonyFernandezMendez2021}, which contains SGE as a
special case. 
The strain-gradient formulations in \cite{EngelEtAl2002,VenturaCodonyFernandezMendez2021}, however, do
not establish error estimates that are simultaneously uniform in
the size parameter $\iota$ and in the nearly
incompressible limit $\lambda\to\infty$.

Our recent works pursued two different parameter-robust approaches to
the SGE model. The $H^2$-nonconforming primal method in \cite{HuangHuangTang2024}
achieves optimal convergence by combining an enriched displacement
space with a Nitsche treatment of the normal-derivative boundary
condition, but its global coupling is not restricted to
$(d-1)$-dimensional faces. The mixed method in \cite{HuangTang2026} lowers the
displacement approximation to the vector-valued Crouzeix--Raviart
element, but introduces a third-order double-stress tensor as an
additional primary variable.

The present work follows a third route. We retain the primal
displacement formulation, introduce no tensor-valued auxiliary
variable, and transfer the continuity requirements of the
fourth-order term to a symmetric interior penalty formulation on the
full mesh skeleton. This permits an $H^1$-nonconforming displacement
approximation whose interelement coupling is entirely face based.
Unlike standard $C^0$ interior penalty discretizations, the present
method relaxes full $H^1$ conformity while preserving normal
continuity across element faces. The resulting displacement space is
therefore $H(\div)$-conforming despite being $H^1$-nonconforming,
and its associated interpolation operator commutes with the divergence.

To the best of our knowledge, no existing primal interior penalty
method for the SGE model, valid in arbitrary dimensions,
simultaneously combines such a displacement space, a
divergence-commuting interpolation, and error estimates with
constants uniform in the limits $\iota\to0$ and
$\lambda\to\infty$.

To realize this approach, we construct, in arbitrary dimensions, a
finite element with a local space $V(T)$ consisting of
vector-valued quadratic polynomials enriched by explicitly
represented divergence-free face bubbles. This space satisfies
\[
\div V(T)=\mathbb P_1(T).
\]
In two and three dimensions, $V(T)$ coincides with the $k=2$
local space in \cite{Guzman2012}. The present construction provides an
explicit representation of the divergence-free face bubbles valid
in arbitrary dimensions, a different set of degrees of freedom,
and a global $H(\div)$-conforming assembly with only face-based
interelement coupling, tailored to the present
structure-preserving IP discretization.
The resulting global space $V_h$ is
$H^1$-nonconforming but $H(\div)$-conforming, with all interelement
continuity constraints associated with $(d-1)$-dimensional faces. It
provides a nonconforming discretization of the lower-regularity short
complex
\begin{equation*}
H_0^1(\Omega;\mathbb R^d)
\xrightarrow{\div}
L_0^2(\Omega)
\longrightarrow 0,
\end{equation*}
through the discrete sequence
\begin{equation*}
V_h
\xrightarrow{\div}
\mathbb P_1(\mathcal T_h)\cap L_0^2(\Omega)
\longrightarrow 0.
\end{equation*}
More precisely, the divergence map is surjective and the associated
interpolation operator
$
I_h:H_0^1(\Omega;\mathbb R^d)\longrightarrow V_h
$
satisfies the commuting relation
\begin{equation}\label{intro-commutative}
\div I_h\boldsymbol v
=Q_{1,h}\div\boldsymbol v
\qquad
\forall\,\boldsymbol v\in H_0^1(\Omega;\mathbb R^d),
\end{equation}
where $Q_{1,h}$ is the elementwise $L^2$-orthogonal projection onto
$\mathbb P_1(\mathcal T_h)$. This commuting property is the central
structural ingredient in the locking-free analysis.

Based on $V_h$, we formulate a symmetric interior penalty
nonconforming finite element method for the primal SGE problem. The
penalty terms are imposed on the full mesh skeleton and control the
jumps of both the elementwise strain and the divergence. The penalty
threshold depends only on the dimension and the shape regularity of the
mesh and is independent of $h$, $\iota$, and $\lambda$. Using the
commuting property~\eqref{intro-commutative}, the interpolation
estimates, and a conforming connection operator, we establish an
$\mathcal O(h)$ energy-norm estimate for fixed $\iota$. For sufficiently
smooth data, we further obtain the complementary parameter-uniform
bound $\mathcal O(\iota^{1/2}+h^2)$,
with a constant independent of both $\iota$ and $\lambda$. Thus, the
method combines a primal formulation and an $H^1$-nonconforming
displacement approximation with optimal convergence and robustness in
both parameter limits.

The same divergence-commuting structure further yields exact
nonconforming finite element Stokes complexes in two and three
dimensions on contractible domains. Related two-dimensional
constructions can be found in
\cite{Mardal2002,Guzman2012,FalkNeilan2013,Guzman2014,
Zhang2016,Guzman2020,chen2022finite}, while three-dimensional
constructions include
\cite{TaiWinther2006,Guzman2012,Neilan2015,MR4621133,
MR4654617,HuangZhang2024}. 

The theoretical convergence and parameter
robustness are examined numerically in both two and three dimensions.
A circular-hole benchmark under uniaxial tension is also considered to
illustrate the applicability of the method to a multiply connected
domain and to compare the computed displacement and Cauchy hoop stress
with the analytical Khakalo--Niiranen solution.

The rest of this paper is organized as follows. Section \ref{sec2} introduces the notation and uniform regularity estimates for the SGE model. Section \ref{sec3} constructs the $H^1$-nonconforming finite element and the associated two- and three-dimensional Stokes complexes. In Section \ref{sec4}, we develop and analyze an optimal and robust IP nonconforming FEM. Finally, Section~\ref{sec5} presents numerical results that verify the theoretical optimality and robustness and demonstrate the applicability of the method to a multiply connected domain.

\section{Preliminaries}\label{sec2}
\subsection{Notation}
Let $\Omega\subset\mathbb{R}^d$ ($d\geq 2$) be a bounded polytope with boundary $\partial\Omega$.
Given a bounded domain $D$ and an integer $m\geq 0$, denote by $H^m(D) $ the standard Sobolev space on $D$ with norm $\|\cdot\|_{m,D}$ and seminorm $|\cdot |_{m, D}$, and $H_0^m( D)$ the closure of $C_0^\infty(D)$ with respect to $\|\cdot\|_{m,D}$. The notation $(\cdot,\cdot)_D$ denotes the $L^2$ inner product on $D$. Referring to \cite{GiraultRaviart1986}, the Sobolev spaces $H(\curl, D)$, $H_0(\curl, D)$, $H(\div, D)$, $H_0(\div, D)$ and $L^2_0(D)$ are defined in the standard way. For $D= \Omega$, we abbreviate $\|\cdot \|_{m, D}$, $|\cdot |_{m, D}$ and $( \cdot , \cdot ) _D$ as $\|\cdot \|_m, |\cdot |_m$ and $(\cdot,\cdot)$, respectively. Denote by $h_D$ the diameter of $D$.

For a $d$-dimensional simplex $T$ with vertices $\texttt{v}_0, \ldots, \texttt{v}_d$, let $\mathcal{F}(T)$, $\mathcal{E}(T)$, $\mathcal{V}(T)$ denote the set of all $(d-1)$-dimensional faces, one-dimensional edges and vertices of $T$, respectively. 
Let $F_i\in\mathcal{F}(T)$ be the $(d-1)$-dimensional face opposite to vertex $\texttt{v}_i$, and $\lambda_i$ be the barycentric coordinate corresponding to vertex $\texttt{v}_i$. Then $\lambda_i(\boldsymbol{x})$ is a linear polynomial and $\lambda_i|_{F_i}=0$.
Let $b_T=\lambda_0 \lambda_1\cdots \lambda_d=\lambda_i b_{F_i}$ be the bubble function of $T$,
where $b_{F_i}$ is the bubble function of $F_i$.

Denote by $\mathcal{T}_h=\{T\}$ a conforming triangulation of $\Omega$ with each element being a simplex, where $h:=\max_{T\in\mathcal{T}_h}h_T$. Let $\mathcal{F}_h$, $\mathring{\mathcal{F}}_h$, $\mathcal{E}_h$, $\mathring{\mathcal{E}}_h$, $\mathcal{V}_h$ and $\mathring{\mathcal{V}}_h$ be the set of all $(d-1)$-dimensional faces, interior $(d-1)$-dimensional faces, edges, interior edges, vertices and interior vertices, respectively. 
For $F\in\mathcal{F}_h$ and $i=1, \ldots, d-1$, let $\boldsymbol{t}^F_i$ be its mutually perpendicular unit tangential vectors and $\boldsymbol{n}_F$ be the unit normal vector. We will abbreviate $\boldsymbol{t}^F_i$ and $\boldsymbol{n}_F$ as $\boldsymbol{t}_i$ and $\boldsymbol{n}$, respectively, if not causing any confusion. In particular, for $e\in\mathcal{E}_h$, we denote its unit tangent vector by $\boldsymbol{t}_e$, abbreviated as $\boldsymbol{t}$.
For a finite set $A$, denote by $|A|$ its cardinality.

Given a face $F\in\mathcal{F}_h$ and a vector $\boldsymbol{v}\in\mathbb{R}^d$, 
let $\mathscr T_F$ denote the tangent space to $F$, and define 
\begin{align*}
\Pi_F\boldsymbol{v} = (\boldsymbol{I}-\boldsymbol{n}_F\boldsymbol{n}^\intercal_F)\boldsymbol{v}
\end{align*}
as the projection of $\boldsymbol{v}$ onto $\mathscr T_F$.
For a $d\times d$ matrix $A$, we denote the skew-symmetric part of $A$ as
$$\skw A:=\dfrac{1}{2}(A-A^\intercal).$$ 
Denote by $\mathbb{K}$ the space of all skew-symmetric $d\times d$ matrices.  

For $F\in\mathring{\mathcal{F}}_h$ shared by two simplices $T^+$ and $T^-$, 
denote by $\boldsymbol{n}^+$ and $\boldsymbol{n}^-$ the unit outward normal to $T^+$ and $T^-$, respectively.
We preset the unit normal vector of $F$ by $\boldsymbol{n}_F=\boldsymbol{n}^+$.
Define the average and jump on $F$ as
$$\{\boldsymbol{v}\}|_F:=\frac{1}{2}(\boldsymbol{v}|_{T^+}+\boldsymbol{v}|_{T^-}), \quad [\![\boldsymbol{v}]\!]|_F:=\boldsymbol{v}|_{T^+}-\boldsymbol{v}|_{T^-}.$$
Set $\{\boldsymbol{v}\}|_F=\boldsymbol{v}|_F$ and $[\![\boldsymbol{v}]\!]|_F=\boldsymbol{v}|_F$ for all $F\in\mathcal F_h\cap\partial\Omega$.


For a bounded domain $D\subset \mathbb{R}^d$ and a non-negative integer $k$, let $\mathbb{P}_k(D)$ stand for the set of all polynomials over $D$ with the total degree no more than $k$. Denote by $Q_{k,D}$ the standard $L^2$ projection operator from $L^2(D)$ to $\mathbb{P}_k(D)$ , whose vectorial/tensorial version is also denoted by $Q_{k,D}$ if there is no confusion.
Let $Q_{k,h}$ be the elementwise version of $Q_{k,D}$ with respect to $\mathcal{T}_h$.
For $s\geq 1$ and integer $k\geq0$, introduce
\begin{align*}
H^s(\mathcal{T}_h)&:=\{v\in L^2(\Omega): v|_T\in H^s(T) \quad \forall \ T\in\mathcal{T}_h\}, \\
\mathbb P_k(\mathcal{T}_h)&:=\{v\in L^2(\Omega): v|_T\in \mathbb P_k(T) \quad \forall \ T\in\mathcal{T}_h\}.	
\end{align*}
Denote 
$$
\mathbb P^c_k(\mathcal{T}_h):=\mathbb P_k(\mathcal{T}_h)\cap H^1(\Omega),\, \mathring{\mathbb P}^c_k(\mathcal{T}_h):=\mathbb P^c_k(\mathcal{T}_h)\cap H^1_0(\Omega),\, \mathring{\mathbb P}_k(\mathcal{T}_h):=\mathbb P_k(\mathcal{T}_h)\cap L^2_0(\Omega).$$
In addition,
for space $\mathbb{B}(D)$, set $\mathbb{B}(D;\mathbb{X}):=\mathbb{B}(D)\otimes\mathbb{X}$ with $\mathbb{X}=\mathbb{R}^d$ or $\mathbb{K}$.

We use $\nabla_h$ and $\boldsymbol{\varepsilon}_h$ to represent the elementwise version of $\nabla$ and $\boldsymbol{\varepsilon}$ with respect to $\mathcal T_h$.
For $\boldsymbol{v}\in H^1(\mathcal{T}_h;\mathbb R^d)\cap H(\div,\Omega)$, denote
\begin{equation*}
\boldsymbol{\sigma}_h(\boldsymbol{v}):=2\mu\boldsymbol{\varepsilon}_h(\boldsymbol{v})+\lambda(\div\boldsymbol{v})\boldsymbol{I}.	
\end{equation*}
For piecewise smooth scalar, vector-valued or tensor-valued function $v$, define the broken squared norm and seminorms
$$
\|v\|_{s,h}^2:=\sum_{T\in\mathcal{T}_h}\|v\|^2_{s,T}, \;\; |v|_{s,h}^2:=\sum_{T\in\mathcal{T}_h}|v|^2_{s,T},\;\;\interleave v\interleave^2_{1,h}:=|v|^2_{1,h}+\sum_{F\in\mathcal F_h}h^{-1}_F\|[\![v]\!]\|^2_{0,F}.$$
For piecewise smooth function $\boldsymbol{v}\in H(\div,\Omega)$, 
we introduce the following discrete norms
\begin{align*}
\|\boldsymbol{v}\|^2_{\iota,\lambda,h}&:=2\mu\|\boldsymbol{\varepsilon}_h(\boldsymbol{v})\|^2_{0}+\lambda\|\div\boldsymbol{v}\|^2_{0}+\iota^2(2\mu|\boldsymbol{\varepsilon}_h(\boldsymbol{v})|^2_{1,h}+\lambda|\div\boldsymbol{v}|^2_{1,h}),\\
\interleave\boldsymbol{v}\interleave^2_{\iota,\lambda,h}&:=2\mu\|\boldsymbol{\varepsilon}_h(\boldsymbol{v})\|^2_{0}+\lambda\|\div\boldsymbol{v}\|^2_{0}+\iota^2(2\mu\interleave\!\boldsymbol{\varepsilon}_h(\boldsymbol{v})\interleave^2_{1,h} + \lambda\interleave\!\div\boldsymbol{v}\interleave^2_{1,h}).
\end{align*}

Throughout the paper, we assume that
\[
0<\mu_0\leq \mu\leq \mu_1,
\qquad
\lambda\geq \lambda_0>0,
\]
where $\mu_0$, $\mu_1$, and $\lambda_0$ are fixed constants.
We use $a\lesssim b$ to mean that
$a\leq Cb$,
where $C>0$ is a generic constant that may depend on
$\mu_0$, $\mu_1$, and $\lambda_0$, but is independent of the mesh size
$h$, the size parameter $\iota$, and the Lam\'e coefficient $\lambda$.
Finally, $a\eqsim b$ means that $a\lesssim b$ and $b\lesssim a$.

Recall some notations in two dimensions. For scalar function $w$ and vector function $\boldsymbol{v} = (v_1, v_2)^\intercal$, denote
\[
\curl w = \left(\partial_2w, -\partial_1w\right)^{\intercal}, \quad 
\rot\boldsymbol{v} = \partial_1v_2 - \partial_2v_1.
\]

\subsection{Weak formulation and regularity results}
The weak formulation of problem \eqref{SGE0} is to find $\boldsymbol{u}\in H^2_0(\Omega;\mathbb{R}^d)$ such that
\begin{align}\label{weak1}
\iota^2a(\boldsymbol{u},\boldsymbol{v})+b(\boldsymbol{u},\boldsymbol{v})=(\boldsymbol{f},\boldsymbol{v}) \quad \forall \ \boldsymbol{v}\in H^2_0(\Omega;\mathbb{R}^d),
\end{align}
where
\begin{equation*}
a(\boldsymbol{u},\boldsymbol{v}):=(\nabla \boldsymbol{\sigma}(\boldsymbol{u}), \nabla \boldsymbol{\varepsilon}(\boldsymbol{v})), \quad
b(\boldsymbol{u},\boldsymbol{v}):=(\boldsymbol{\sigma}(\boldsymbol{u}), \boldsymbol{\varepsilon}(\boldsymbol{v})).
\end{equation*}

Taking $\iota=0$,  problem \eqref{SGE0} becomes the linear elasticity problem
\begin{equation}\label{SGElinear}
\begin{cases}
-2\mu\div(\boldsymbol{\varepsilon}(\boldsymbol{u}_0))-\lambda \nabla\div\boldsymbol{u}_0=\boldsymbol{f} &\mbox{in} \ \Omega,\\
\boldsymbol{u}_0=0 &\mbox{on} \ \partial\Omega.
\end{cases}
\end{equation}
We assume the linear elasticity problem \eqref{SGElinear} has the $s$-regularity with $2\leq s\leq 3$
\begin{equation}\label{elasregularity}
\|\boldsymbol{u}_0\|_s+\lambda\|\div\boldsymbol{u}_0\|_{s-1}\lesssim\|\boldsymbol f\|_{s-2}.
\end{equation}
When $\Omega$ is convex in two and three dimensions, the regularity \eqref{elasregularity} for $s=2$ can be found in \cite{brenner1992linear,MR2641539}.

In this paper, we always assume problem \eqref{SGE0} has the regularity
\begin{align}
\label{Regularity-u}
|\boldsymbol{u}-\boldsymbol{u}_0|_1+\iota\|\boldsymbol{u}\|_2+\iota^2\|\boldsymbol{u}\|_3&\lesssim\iota^{1/2}\|\boldsymbol{f}\|_0,\\
\label{Regularity-divu}
\lambda\|\div(\boldsymbol{u}-\boldsymbol{u}_0)\|_0+\lambda\iota|\div\boldsymbol{u}|_1+\lambda\iota^2\|\div\boldsymbol{u}\|_2&\lesssim\iota^{1/2}\|\boldsymbol{f}\|_0.
\end{align}
When $\Omega$ is convex, the regularity results \eqref{Regularity-u}--\eqref{Regularity-divu} were proved in \cite{ChenHuangHuang2023} under the following assumptions:
\begin{assumption}
\label{assumption1}
Assume that $\boldsymbol{f}\in H^{-1}(\Omega;\mathbb{R}^d)$. Let $(\boldsymbol{\phi},p)\in H^2_0(\Omega;\mathbb{R}^d)\times(H^1_0(\Omega)\cap L^2_0(\Omega))$ be the solution of the following problem
\begin{align*}
\left\{
\begin{array}{ll}
\Delta^2\boldsymbol{\phi}+\nabla\Delta p=\boldsymbol{f} &\mbox{in} \ \Omega,\\
\Delta\div\boldsymbol{\phi}=0 &\mbox{in} \ \Omega,\\
\boldsymbol{\phi}=\partial_{n}\boldsymbol{\phi}=\boldsymbol{0} &\mbox{on} \ \partial\Omega.
\end{array}
\right.
\end{align*}
Then $\boldsymbol{\phi}\in H^3(\Omega;\mathbb{R}^d)$ and $p\in H^2(\Omega)$, which admit the estimate
\begin{equation*}
\|\boldsymbol{\phi}\|_3+\|p\|_2\lesssim\|\boldsymbol{f}\|_{-1}.
\end{equation*}
\end{assumption}
\begin{assumption}
\label{assumption2}
Assume that $\boldsymbol{f}\in H^{-1}(\Omega;\mathbb{R}^d)$, Lam\'{e} coefficients $\lambda\in[0,\Lambda]$ and $\mu\in[\mu_0,\mu_1]$ with $\Lambda,\mu_0,\mu_1>0$. Let $\boldsymbol{u}\in H^{2}(\Omega;\mathbb{R}^d)$ be the solution of the following problem:
\begin{align*}
\left\{
\begin{array}{ll}
\div\Delta(2\mu\boldsymbol{\varepsilon}(\boldsymbol{u})+\lambda(\div\boldsymbol{u})\boldsymbol{I})=\boldsymbol{f} &\mbox{in} \ \Omega,\\
\boldsymbol{u}=\partial_{n}\boldsymbol{u}=\boldsymbol{0} &\mbox{on} \ \partial\Omega.
\end{array}
\right.
\end{align*}
Then $\boldsymbol{u}\in H^{3}(\Omega;\mathbb{R}^d)$ and it admits the estimate
$$\|\boldsymbol{u}\|_3\lesssim\|\boldsymbol{f}\|_{-1},$$
where the hidden constant may depend on $\Lambda$, $\mu_0$ and $\mu_1$.
\end{assumption}

Recall the following regularity results of the fourth-order problem in \cite[Lemma 3.7]{ChenHuangHuang2023}.
\begin{lemma}
Let $\boldsymbol{w}\in H^2_0(\Omega;\mathbb{R}^d)$ be the solution of 
\begin{align*}
\left\{
\begin{array}{ll}
\div\Delta(2\mu\boldsymbol{\varepsilon}(\boldsymbol{w})+\lambda(\div\boldsymbol{w})\boldsymbol{I})=\boldsymbol{f} &\mbox{in} \ \Omega,\\
\boldsymbol{w}=\partial_{n}\boldsymbol{w}=\boldsymbol{0} &\mbox{on} \ \partial\Omega,
\end{array}
\right.
\end{align*}
with $\boldsymbol{f}\in H^{-1}(\Omega;\mathbb{R}^d)$. Suppose Assumptions \ref{assumption1}--\ref{assumption2} hold. We have
\begin{align}
\label{fourregularity1}
&\|\boldsymbol{w}\|_{2+j}+\lambda\|\div\boldsymbol{w}\|_{1+j}\lesssim\|\boldsymbol{f}\|_{j-2} \quad \mbox{for}~j=0,1,\\
\label{fourregularity2}
&\|\boldsymbol{w}\|_3\lesssim\inf_{q\in L^2(\Omega)}\|\boldsymbol{f}+\nabla q\|_{-1}+\dfrac{1}{\lambda}\|\boldsymbol{f}\|_{-1}.
\end{align}
\end{lemma}

Next, we refine the regularity results \eqref{Regularity-u}--\eqref{Regularity-divu} by replacing the right-hand side term $\iota^{1/2}\|\boldsymbol{f}\|_0$ with $\|\boldsymbol{f}\|_{-1}$.
\begin{lemma}
Suppose Assumptions \ref{assumption1}--\ref{assumption2} hold. Let $\boldsymbol{u}\in H^2_0(\Omega;\mathbb{R}^d)$ be the solution of the problem~\eqref{SGE0}. It holds
\begin{equation}
\label{Regularity-u-2}
\iota^2(\|\boldsymbol{u}\|_3+\lambda\|\div\boldsymbol{u}\|_2)+\iota(\|\boldsymbol{u}\|_2+\lambda\|\div\boldsymbol{u}\|_1)+\|\boldsymbol{u}\|_1+\lambda\|\div\boldsymbol{u}\|_0\lesssim\|\boldsymbol{f}\|_{-1}.
\end{equation}
\end{lemma}
\begin{proof}
From \eqref{weak1}, it holds 
\begin{equation*}
2\mu(\iota^2|\boldsymbol{\varepsilon}(\boldsymbol{u})|^2_1+\|\boldsymbol{\varepsilon}(\boldsymbol{u})\|^2_0)\leq (\boldsymbol{f},\boldsymbol{u})\leq \|\boldsymbol{f}\|_{-1}\|\boldsymbol{u}\|_1,
\end{equation*}
which implies
\begin{equation}
\label{Regularity-u-2-pf1}
\iota\|\boldsymbol{u}\|_2+\|\boldsymbol{u}\|_1\lesssim\|\boldsymbol{f}\|_{-1}.
\end{equation}
By \eqref{SGE0}, we have
\begin{equation}\label{fourth}
\div\Delta\boldsymbol{\sigma}(\boldsymbol{u}) = \iota^{-2}(\boldsymbol{f}+\div\boldsymbol{\sigma}(\boldsymbol{u})).
\end{equation}
Then it follows from \eqref{fourregularity1}--\eqref{fourregularity2} and \eqref{Regularity-u-2-pf1} that
\begin{align}
\label{Regularity-u-2-pf2}
&\iota^2(\|\boldsymbol{u}\|_3+\lambda\|\div\boldsymbol{u}\|_2)\lesssim \|\boldsymbol{f}\|_{-1}+\|\boldsymbol{u}\|_1+\lambda\|\div\boldsymbol{u}\|_0\lesssim \|\boldsymbol{f}\|_{-1}+\lambda\|\div\boldsymbol{u}\|_0,\\
\label{Regularity-u-2-pf3}
&\iota^2\|\boldsymbol{u}\|_3\lesssim (1+\dfrac{1}{\lambda})(\|\boldsymbol{f}\|_{-1}+\|\boldsymbol{u}\|_{1})+\|\div\boldsymbol{u}\|_0\lesssim\|\boldsymbol{f}\|_{-1}+\|\boldsymbol{u}\|_{1}\lesssim\|\boldsymbol{f}\|_{-1}.
\end{align}
According to \cite{MR2609313}, there exists a $\boldsymbol{v}\in H^2_0(\Omega;\mathbb{R}^d)$ such that 
\begin{equation}
\label{Regularity-u-2-pf4}
\div\boldsymbol{v} = \div\boldsymbol{u}, \quad \quad |\boldsymbol{v}|_1\lesssim \|\div\boldsymbol{u}\|_0, \quad \quad |\boldsymbol{v}|_2\lesssim |\div\boldsymbol{u}|_1.
\end{equation}
Multiply \eqref{fourth} by $\boldsymbol{v}$, and apply the integration by parts and the Cauchy--Schwarz inequality to get
\begin{align*}
&\lambda\iota^2|\div\boldsymbol{u}|^2_1+\lambda\|\div\boldsymbol{u}\|^2_0=\lambda\iota^2(\nabla\div\boldsymbol{u},\nabla\div\boldsymbol{v})+\lambda(\div\boldsymbol{u},\div\boldsymbol{v})\\
=&(\boldsymbol{f},\boldsymbol{v})+2\mu\big(\iota^2(\Delta\boldsymbol{\varepsilon}(\boldsymbol{u}),\boldsymbol{\varepsilon}(\boldsymbol{v}))-(\boldsymbol{\varepsilon}(\boldsymbol{u}),\boldsymbol{\varepsilon}(\boldsymbol{v}))\big)\lesssim(\|\boldsymbol{f}\|_{-1}+\iota^2|\boldsymbol{u}|_3+|\boldsymbol{u}|_1)|\boldsymbol{v}|_1,
\end{align*}
which together with \eqref{Regularity-u-2-pf1} and \eqref{Regularity-u-2-pf3}--\eqref{Regularity-u-2-pf4} yields
\begin{equation*}
\lambda\iota^2|\div\boldsymbol{u}|^2_1+\lambda\|\div\boldsymbol{u}\|^2_0\lesssim \|\boldsymbol{f}\|_{-1}\|\div\boldsymbol{u}\|_0.
\end{equation*}
Hence, it follows
\begin{equation}
\label{Regularity-u-2-pf5}
\lambda\iota|\div\boldsymbol{u}|_1+\lambda\|\div\boldsymbol{u}\|_0\lesssim \|\boldsymbol{f}\|_{-1}.
\end{equation}
Therefore, \eqref{Regularity-u-2} holds from \eqref{Regularity-u-2-pf1}, \eqref{Regularity-u-2-pf2} and \eqref{Regularity-u-2-pf5}.
\end{proof}
\section{Discrete Spaces and Complexes}\label{sec3}
In this section we will construct an $H^1$-nonconforming quadratic finite element and develop the corresponding interpolation operator along with its interpolation error estimates. Moreover, we establish nonconforming finite element Stokes complexes in two and three dimensions. For the latter constructions, we additionally assume that \(\Omega\) is contractible.

\subsection{$H^1$-nonconforming finite element and interpolation operator}\label{NCFEM}
Let $T$ be a $d$-dimensional simplex with $d\geq 2$. With the notation introduced above, take
\begin{align*}
V(T):=&\mathbb{P}_{2}(T;\mathbb{R}^{d})
\oplus \Oplus_{i=0}^{d}\Oplus_{j=1}^{d-1}\div(\skw(\boldsymbol{n}_{F_i}\otimes\boldsymbol{t}_j^{F_i})b_{T}b_{F_i}\mathbb{P}_{1}(F_i))
\end{align*}
as the space of shape functions. Clearly, $V(T)\subset \mathbb{P}_{2d+1}(T;\mathbb{R}^{d})$, $\div V(T) = \mathbb{P}_{1}(T)$, and $(V(T)\cdot\boldsymbol{n})|_{F_i}\in\mathbb P_2(F_i)$ for $i=0,1,\ldots,d$.
It is easy to see that 
$$\dim V(T)=d\left(\begin{matrix}d+2\\2\end{matrix}\right)+(d-1)d(d+1)=\frac{3}{2}d^2(d+1).$$
\begin{remark}\rm
In two dimensions,
\begin{align*}
\div\big(\skw(\boldsymbol{n}_{F_i}\otimes\boldsymbol{t}_j^{F_i})b_{T}b_{F_i}\mathbb{P}_{1}(F_i)\big)=\curl(b_{T}b_{F_i}\mathbb{P}_{1}(F_i)).
\end{align*}
In three dimensions,
\begin{align*}
\div\big(\skw(\boldsymbol{n}_{F_i}\otimes\boldsymbol{t}_j^{F_i})b_{T}b_{F_i}\mathbb{P}_{1}(F_i)\big)=\curl(b_{T}b_{F_i}\mathbb{P}_{1}(F_i)\boldsymbol{t}_{3-j}^{F_i}).
\end{align*}
\end{remark}
The degrees of freedom (DoFs) are chosen as
\begin{subequations}\label{dof}
\begin{align}
\label{dof1}
\int_{F}\boldsymbol{v}\cdot \boldsymbol{n} \ q \ds,& \quad \ q\in \mathbb{P}_2(F), F\in\mathcal{F}(T),\\
\label{dof2}
\int_{F}\boldsymbol{v}\cdot\boldsymbol{t}_i \ q \ds,& \quad \ q\in\mathbb{P}_1(F), F\in\mathcal{F}(T), i = 1,2,\cdots,d-1,\\
\label{dof3}
\int_{T}\boldsymbol{v}\cdot\boldsymbol{q} \dx,& \quad \ \boldsymbol{q}\in\mathbb{P}_0(T;\mathbb{R}^d),\\
\label{dof4}
\dfrac 1{d+1}\sum_{i=0}^{d}(\skw\nabla\boldsymbol{v})(\texttt{v}_{i}).&
\end{align}
\end{subequations} 
\begin{lemma}
\label{dof-uni-solvent}
The DoFs  \eqref{dof} are unisolvent for $V(T)$.
\end{lemma}
\begin{proof}
The number of the DoFs \eqref{dof} is
$$(d+1)\left(\begin{matrix}d+1\\2\end{matrix}\right)+(d-1)d(d+1)+d+\dfrac{(d-1)d}{2}=\frac{3}{2}d^2(d+1)=\dim V(T).$$
Take $\boldsymbol{v}\in V(T)$ and assume all the DoFs \eqref{dof} vanish,
then we prove $\boldsymbol{v}=\boldsymbol{0}$.
By the vanishing DoF \eqref{dof1}, 
we get $\boldsymbol{v}\cdot\boldsymbol{n}|_{\partial T}=0$. Hence $\div\boldsymbol{v}\in L_0^2(T)$. Together with the vanishing DoF \eqref{dof3} and the fact $\mathbb{P}_0(T;\mathbb{R}^d)=\nabla \mathbb{P}_1(T)$,
we obtain $\div\boldsymbol{v}=0$. Then we can write $\boldsymbol{v}=\boldsymbol{v}_1+\boldsymbol{v}_2$ with $\boldsymbol{v}_1\in\mathbb{P}_{2}(T;\mathbb{R}^{d})$ satisfying $\div\boldsymbol{v}_1=0$ and $\boldsymbol{v}_2\in\sum_{i=0}^{d}\sum_{j=1}^{d-1}\div\big(\skw(\boldsymbol{n}_{F_i}\otimes\boldsymbol{t}_j^{F_i})b_{T}b_{F_i}\mathbb{P}_{1}(F_i)\big)$.
Thanks to Lemma 3.4 and Lemma 3.5 in \cite{HuangHuangTang2024}, we get from the vanishing DoFs \eqref{dof4} and \eqref{dof2} that $\boldsymbol{v}_1 = 0$ and $\boldsymbol{v}_2 = 0$, respectively.
\end{proof}
\begin{remark}
In two and three dimensions, $V(T)$ is the same as the shape function space of the $H^1$-nonconforming element proposed in \cite{Guzman2012} for $k=2$, but the DoFs differ, particularly for \eqref{dof4}. This discrepancy arises because the bubble functions in \cite{Guzman2012} are defined implicitly by enforcing the vanishing first-order moment, whereas in this paper, the bubble functions are explicitly constructed and the first-order moment does not vanish.
\end{remark}

Define the global $H^1$-nonconforming finite element space:
\begin{align*}
V_{h}&:=\{\boldsymbol{v}_{h}\in L^{2}(\Omega;\mathbb{R}^{d}):\boldsymbol{v}_h|_T\in V(T)~\textrm{for~each}~T\in\mathcal{T}_h;
~\textrm{DoFs}~\eqref{dof1}\text{-}\eqref{dof2}\\
&\qquad\quad\;\;\textrm{are~single-valued},
\textrm{ and DoFs}~\eqref{dof1}\text{-}\eqref{dof2}~\textrm{vanish on boundary}\}.
\end{align*}
Clearly $ V_{h}\subset H_0(\div,\Omega)$, but $V_{h}\not\subseteq H_0^1(\Omega;\mathbb R^d)$.
Notice that the finite element space $ V_h$ has the weak continuity
\begin{align}
\label{weak-continuity-vh1}
\int_{F}[\![\boldsymbol{v}_h]\!]\cdot \boldsymbol{q} \ds=0 \quad \forall \ \boldsymbol{v}_h\in V_h, \ \boldsymbol{q}\in\mathbb{P}_1(F;\mathbb{R}^d), \ F\in\mathcal{F}_h.
\end{align}
Then we define the following interpolation operator $I_h: H^1_0(\Omega;\mathbb{R}^d)\rightarrow V_h$ as follows:
for any $\boldsymbol{v}\in H^1_0(\Omega;\mathbb{R}^d),$
\begin{align}
\label{interpolation-I1}&\int_{F}(I_h\boldsymbol{v})\cdot\boldsymbol{n} \ q\ds=\int_{F}\boldsymbol{v}\cdot\boldsymbol{n} \ q\ds \quad \forall \ q\in \mathbb{P}_2(F), F\in\mathring{\mathcal{F}}_h,\\
\notag
&\int_{F}\Pi_F(I_h\boldsymbol{v}) \cdot \boldsymbol{q} \ds=\int_{F}\Pi_F\boldsymbol{v} \cdot \boldsymbol{q}\ds\quad \forall\,\boldsymbol{q}\in\mathbb{P}_1(F;\mathscr T_F), F\in\mathring{\mathcal{F}}_h,\\
\label{interpolation-I3}&\int_{T}(I_h\boldsymbol{v})\cdot\boldsymbol{q} \dx=\int_{T}\boldsymbol{v}\cdot\boldsymbol{q} \dx \quad \forall \ \boldsymbol{q}\in \mathbb{P}_0(T;\mathbb{R}^d), T\in\mathcal{T}_h,\\
\notag
&\frac 1{d+1}\sum_{i=0}^{d}(\skw\nabla(I_h\boldsymbol{v}))(\texttt{v}_{i})=\frac 1{d+1}\sum_{i=0}^{d}(\skw\nabla(Q_{2,T}\boldsymbol{v}))(\texttt{v}_{i}) \quad \forall \ T\in\mathcal{T}_h.
\end{align}
\begin{lemma}
\label{interpolation-I-lemma1}
For integers $0\leq m\leq 3$ and $j=0,1$, we have for any $\boldsymbol{v}\in H^1_0(\Omega;\mathbb{R}^d)\cap H^s(\Omega;\mathbb{R}^d)$ and $T\in\mathcal{T}_h$ that
\begin{equation}
\label{Ih-error1}
|\boldsymbol{v}-I_h\boldsymbol{v}|_{m,T}\lesssim h_T^{s-m}|\boldsymbol{v}|_{s,T} \quad \textrm{ for } \max\{m,1\}\leq s\leq 3,
\end{equation}
\begin{equation}
\label{Ih-error11}
\|\partial_n^j\boldsymbol{\varepsilon}(\boldsymbol{v}-I_h\boldsymbol{v})\|_{0,\partial T}\lesssim h_T^{s-j-3/2}|\boldsymbol{v}|_{s,T} \quad \textrm{ for } j+2\leq s\leq 3.
\end{equation}
\end{lemma}
\begin{proof}
Let $\boldsymbol{w}=(I_h\boldsymbol{v})|_T-Q_{2,T}\boldsymbol{v}$ for simplicity.
For any $\boldsymbol{v}\in H^1_0(\Omega;\mathbb{R}^d)$, it follows from the inverse inequality, trace inequality, and scaling argument that
\begin{align*}
h^{2m}_T\|\boldsymbol{w}\|^2_{m,T} \lesssim \|\boldsymbol{w}\|^2_{0,T} \lesssim h_T\sum_{F\in\mathcal{F}(T)}\|\boldsymbol{v}-Q_{2,T}\boldsymbol{v}\|^2_{0,F}+\|\boldsymbol{v}-Q_{2,T}\boldsymbol{v}\|^2_{0,T}.
\end{align*}
Then we derive the estimate \eqref{Ih-error1} from the triangle inequality and the error estimate of $Q_{2,T}$.

By the trace inequality,
\begin{equation*}
\|\partial_n^j\boldsymbol{\varepsilon}(\boldsymbol{v}-I_h\boldsymbol{v})\|_{0,\partial T}\lesssim h_T^{-1/2}|\boldsymbol{\varepsilon}(\boldsymbol{v}-I_h\boldsymbol{v})|_{j,T} + h_T^{1/2}|\boldsymbol{\varepsilon}(\boldsymbol{v}-I_h\boldsymbol{v})|_{j+1,T}.
\end{equation*}
Thus, \eqref{Ih-error11} holds from \eqref{Ih-error1}.
\end{proof}
\begin{lemma}
\label{commutative}
The following commuting property holds:
\begin{align}
\label{commutative1}
\div (I_h\boldsymbol{v})=Q_{1,h}(\div\boldsymbol{v}) \quad \forall~\boldsymbol{v}\in H^1_0(\Omega;\mathbb{R}^d).
\end{align}
\end{lemma}
\begin{proof}
Using the integration by parts, \eqref{interpolation-I1} and \eqref{interpolation-I3}, we get for any $T\in \mathcal{T}_h$ that
\begin{equation*}
(\div(I_h\boldsymbol{v})-Q_{1,h}(\div\boldsymbol{v}),q)_T = (\div(I_h\boldsymbol{v}-\boldsymbol{v}),q)_T = 0 \quad \forall~q\in\mathbb{P}_1(T).
\end{equation*}
Hence \eqref{commutative1} holds from the last equality.
\end{proof}

\begin{corollary}
The following surjectivity property holds:
\begin{equation}\label{eq:div-stability}
\div V_h=\mathring{\mathbb{P}}_{1}(\mathcal{T}_h).
\end{equation}
\end{corollary}
\begin{proof}
Clearly, $\div V_h\subseteq \mathring{\mathbb{P}}_{1}(\mathcal{T}_h)$.
For $q_h\in \mathring{\mathbb{P}}_{1}(\mathcal{T}_h)\subset L_0^2(\Omega)$, there exists a $\boldsymbol{v}\in H^1_0(\Omega;\mathbb{R}^d)$ such that $\div\boldsymbol{v}=q_h$. Then $I_h\boldsymbol{v}\in V_h$, and by \eqref{commutative1} we have $\div(I_h\boldsymbol{v})=Q_{1,h}(\div\boldsymbol{v})=Q_{1,h}q_h=q_h$.
\end{proof}

\begin{lemma}
\label{interpolation-Jh-lemma}
For integers $0\leq m\leq 2$ and $j=0,1$, we have for any $\boldsymbol{v}\in H^1_0(\Omega;\mathbb{R}^d)$ satisfying $\div\boldsymbol{v}\in H^{s}(\Omega)$ that
\begin{equation}
\label{Jh-error1}
|\div\boldsymbol{v}-\div(I_h\boldsymbol{v})|_{m,T}\lesssim h_T^{s-m}|\div\boldsymbol{v}|_{s,T}
\quad \textrm{ for } m\leq s\leq 2, 
\end{equation}
\begin{equation}
\label{Jh-error11}
\|\partial_n^j\div(\boldsymbol{v}-I_h\boldsymbol{v})\|_{0,\partial T}\lesssim h_T^{s-j-1/2}|\div\boldsymbol{v}|_{s,T}
\quad \textrm{ for } j+1\leq s\leq 2. 
\end{equation}
\end{lemma}
\begin{proof}
From the commuting property \eqref{commutative1} and the error estimate of $Q_{1,h}$, we get~\eqref{Jh-error1}. Applying the trace inequality, \eqref{Jh-error11} holds from \eqref{Jh-error1}.
\end{proof}
\begin{lemma}
\label{interpolation-I-lemma2}
Let $\boldsymbol{u}\in H_0^2(\Omega;\mathbb{R}^d)\cap H^3(\Omega;\mathbb{R}^d)$ be the solution of problem \eqref{SGE0}, and the solution $\boldsymbol{u}_0\in H^1_0(\Omega;\mathbb{R}^d)$ of problem \eqref{SGElinear} satisfy the regularity \eqref{elasregularity} with $2\leq s\leq 3$.
We have
\begin{align}
\label{Ih-error2}
\interleave\boldsymbol{u}-I_h\boldsymbol{u}\interleave_{\iota,\lambda,h} & \lesssim \iota h(|\boldsymbol{u}|_3+\sqrt{\lambda}|\div\boldsymbol{u}|_2)+ h(|\boldsymbol{u}|_2+\sqrt{\lambda}|\div\boldsymbol{u}|_1), \\ 
\label{Ih-error3}
\interleave\boldsymbol{u}-I_h\boldsymbol{u}\interleave_{\iota,\lambda,h} & \lesssim\iota^{1/2}\|\boldsymbol{f}\|_0+h^{s-1}\|\boldsymbol{f}\|_{s-2}.
\end{align}
\end{lemma}
\begin{proof}
Using \eqref{Ih-error1}--\eqref{Ih-error11} and \eqref{Jh-error1}--\eqref{Jh-error11}, we acquire the estimate \eqref{Ih-error2},
\begin{align}\begin{split}
\label{interpolation-I-lemma2-pf}
\|\boldsymbol{\sigma}_h(\boldsymbol{u}-I_h\boldsymbol{u})\|_{0}&\leq \|\boldsymbol{\sigma}_h((\boldsymbol{u}-\boldsymbol{u}_0)-I_h(\boldsymbol{u}-\boldsymbol{u}_0))\|_{0}+\|\boldsymbol{\sigma}_h(\boldsymbol{u}_0-I_h\boldsymbol{u}_0)\|_{0} \\
&\lesssim \|\boldsymbol{\sigma}(\boldsymbol{u}-\boldsymbol{u}_0)\|_{0}+h^{s-1}|\boldsymbol{u}_0|_s+\lambda h^{s-1}|\div\boldsymbol{u}_0|_{s-1},
\end{split}\end{align}
and
\begin{align*}
\interleave\boldsymbol{\sigma}_h(\boldsymbol{u}-I_h\boldsymbol{u})\interleave_{1,h}
&\lesssim \interleave\boldsymbol{\varepsilon}_h(\boldsymbol{u}-I_h\boldsymbol{u})\interleave_{1,h} +
\lambda\interleave\div(\boldsymbol{u}-I_h\boldsymbol{u})\interleave_{1,h} \\
&\lesssim|\boldsymbol{u}|_2+\lambda|\div\boldsymbol{u}|_1.
\end{align*}
Then \eqref{Ih-error3} holds from the last two inequalities and \eqref{elasregularity}--\eqref{Regularity-divu}.
\end{proof}
\subsection{Nonconforming finite element Stokes complex in two dimensions}
First, we present a nonconforming discretization of the Stokes complex in two dimensions
\begin{equation}
\label{Stokescomplex2d}
0\xrightarrow{\subset} H_0^2(\Omega)\xrightarrow{\curl}  H_0^1(\Omega;\mathbb{R}^2)
\xrightarrow{\div} L_0^2(\Omega) \xrightarrow {}0.
\end{equation}
By the div surjectivity \eqref{eq:div-stability}, we only focus on the construction of the discrete space of $H_0^2(\Omega)$ in complex \eqref{Stokescomplex2d}.

Take the space of shape functions
\begin{align*}
W(T):=\mathbb{P}_{3}(T)
\oplus\Oplus_{i=0}^{2}(b_{T}b_{F_i}\mathbb{P}_{1}(F_i)).
\end{align*}
Then $\dim W(T) = 16$.
The DoFs for $W(T)$ are given by
\begin{subequations}\label{2d-wdof}
\begin{align}
\label{2d-wdof1}
v(\texttt{v}), & \quad \texttt{v}\in\mathcal{V}(T),\\
\label{2d-wdof2}
\int_{F}v \, q \ds, \quad \int_{F}\partial_nv \, q \ds, & \quad q\in\mathbb{P}_1(F), F\in\mathcal{F}(T),\\
\label{2d-wdof3}
\frac 13\sum_{i=0}^{2}(\Delta v)(\texttt{v}_{i}).&
\end{align}
\end{subequations}
\begin{lemma}
\label{2d-w-dof-uni-solvent}
The DoFs \eqref{2d-wdof} are unisolvent for $W(T)$.
\end{lemma}
\begin{proof}
The number of the DoFs \eqref{2d-wdof} is
$3+6+6+1=16=\dim W(T).$

Take $ v\in W(T)$ and assume all the DoFs \eqref{2d-wdof} vanish,
then we prove $v=0$. Notice that $v|_F\in\mathbb{P}_{3}(F)$ for each $F\in\mathcal{F}(T)$, and hence we get from the vanishing DoFs \eqref{2d-wdof1}--\eqref{2d-wdof2} that $v|_{\partial T} = 0$. 
Then $\curl v\in H_0(\div,T)$. By comparing DoFs~\eqref{2d-wdof} with DoFs \eqref{dof}, we conclude $\curl v=0$ from Lemma~\ref{dof-uni-solvent}, the vanishing DoFs~\eqref{2d-wdof2}--\eqref{2d-wdof3}, $(\curl v)\cdot\boldsymbol{t}=-\partial_nv$, and $\rot(\curl v)=-\Delta v$.
Thus, $v=0$.
\end{proof}
Define the global $H^2$-nonconforming finite element space:
\begin{align*}
W_{h}&=\{v_{h}\in L^{2}(\Omega):v_h|_T\in W(T)~\textrm{for}~T\in\mathcal{T}_h; \textrm{all the DoFs \eqref{2d-wdof} are} \\
&\qquad\quad\textrm{single-valued, and DoFs \eqref{2d-wdof1}--\eqref{2d-wdof2} vanish on boundary}\}.
\end{align*}
We have $W_h\subset H_0^1(\Omega)$ and $W_h\not\subseteq H_0^2(\Omega)$.
\begin{lemma}
The nonconforming discrete complex
\begin{equation}
\label{2dcomplex-2}
0\xrightarrow{\subset} W_h\xrightarrow{\curl}  V_h
\xrightarrow{\div} \mathring{\mathbb{P}}_{1}(\mathcal{T}_h) \xrightarrow {}0
\end{equation}
is exact.
\end{lemma}
\begin{proof}
Clearly, sequence \eqref{2dcomplex-2} is a complex.

For $\boldsymbol{v}_h\in V_h\cap\ker(\div)\subset H_0(\div,\Omega)\cap\ker(\div)$,
it follows that $\boldsymbol{v}_h=\curl w_h$ with $w_h\in H_0^1(\Omega)$ satisfying $w_h|_T\in W(T)$ for $T\in\mathcal T_h$. Then $w_h\in W_h$ holds from the single-valued DoF \eqref{dof2}. Hence, we have $V_h\cap\ker(\div)=\curl W_h$.
This together with \eqref{eq:div-stability} ends the proof.
\end{proof}

Define interpolation operator $I_h^W: H_0^2(\Omega)\rightarrow W_h$ as follows:
\begin{align*}
(I_h^Wv)(\texttt{v})&=v(\texttt{v}) \qquad\qquad\qquad\qquad\qquad\,\forall~\texttt{v}\in \mathring{\mathcal{V}}_h, \\
\int_{F}(I_h^Wv)\,q\ds&=\int_{F}v\,q\ds \qquad\qquad\qquad\qquad\, \forall~q\in\mathbb{P}_1(F), F\in\mathring{\mathcal{F}}_h,\\
\int_{F}\curl(I_h^Wv) \cdot \boldsymbol{t}\,q\ds&=\int_{F}I_h(\curl v) \cdot \boldsymbol{t}\,q\ds \qquad\quad\;\, \forall~q\in\mathbb{P}_1(F), F\in\mathring{\mathcal{F}}_h,\\
\sum_{i=0}^{2}(\Delta(I_h^Wv))(\texttt{v}_{i})&=-\sum_{i=0}^{2}(\rot(I_h(\curl v)))(\texttt{v}_{i}) \quad \forall~T\in\mathcal{T}_h.
\end{align*}
It can be verified that 
\begin{equation}\label{commutative2dcurl}
\curl(I_h^Wv)=I_h(\curl v)\quad\forall~ v\in H_0^2(\Omega).
\end{equation}

The combination of complex \eqref{2dcomplex-2}, the commuting properties \eqref{commutative2dcurl} and \eqref{commutative1}, yields the following commutative diagram
$$
\xymatrix{
0 \ar[r]^{\subset\hspace{1em}} &
H_0^2(\Omega) \ar[r]^{\curl\hspace{1em}} \ar[d]^{I^W_h} &
H_0^1(\Omega;\mathbb{R}^2) \ar[r]^{\hspace{0.5em}\div} \ar[d]^{I_h} &
L_0^2(\Omega) \ar[r] \ar[d]^{Q_{1,h}} & 0\\
0 \ar[r]^{\subset} & W_h \ar[r]^{\curl} &
V_h \ar[r]^{\div\hspace{1em}} & \mathring{\mathbb{P}}_{1}(\mathcal{T}_h) \ar[r] & 0. 
}
$$

\subsection{Nonconforming finite element Stokes complex in three dimensions}
Recall the Stokes complex in three dimensions
\begin{equation}
\label{Stokescomplex3d}
0\xrightarrow{\subset} H_0^1(\Omega)\xrightarrow{\nabla}  H_0(\grad\curl,\Omega)\xrightarrow{\curl} H_0^1(\Omega;\mathbb{R}^3)
\xrightarrow{\div} L^2_0(\Omega) \xrightarrow {}0,
\end{equation}
with
\begin{equation*}
H_0(\grad\curl,\Omega):=\{\boldsymbol{v}\in H_0(\curl,\Omega): \curl\boldsymbol{v}\in H_0^1(\Omega;\mathbb{R}^3)\}.
\end{equation*}

To discretize space $H_0(\grad\curl,\Omega)$, take the space of shape functions
\begin{align*}
W(T):=\mathbb{P}_{3}^-(T;\mathbb{R}^{3})
&\oplus\Oplus_{i=0}^{3}(b_{T}b_{F_i}\mathbb{P}_{1}(F_i)\otimes\mathscr T^{F_i}),
\end{align*}
where $\mathbb{P}_{3}^-(T;\mathbb{R}^{3}):=\nabla\mathbb{P}_{3}(T)\oplus(\mathbb{P}_{2}(T;\mathbb R^3)\times\boldsymbol{x})$, and $\mathscr T^{F_i}$ is the tangent plane of $F_i$.
We have $\dim W(T) = 69$.
The DoFs for $W(T)$ are given by
\begin{subequations}\label{3d-wdof}
\begin{align}
\label{3d-wdof1}
\int_{e}\boldsymbol{v} \cdot\boldsymbol{t}\,q \ds, & \quad q\in\mathbb{P}_2(e), e\in\mathcal{E}(T),\\
\label{3d-wdof2}
\int_{F}(\Pi_F\boldsymbol{v}) \cdot \boldsymbol{q} \ds, & \quad \boldsymbol{q}\in\boldsymbol{x}\mathbb{P}_0(F), F\in\mathcal{F}(T),\\
\label{3d-wdof3}
\int_{F}\curl\boldsymbol{v} \cdot\boldsymbol{n} \, q \ds, & \quad q\in\mathbb{P}_2(F)/\mathbb{R}, F\in\mathcal{F}(T),\\
\label{3d-wdof4}
\int_{F}(\curl\boldsymbol{v}) \cdot \boldsymbol{t}_i \, q \ds, & \quad q\in\mathbb{P}_1(F), F\in\mathcal{F}(T), i=1,2,\\
\label{3d-wdof5}
\frac{1}{4}\sum_{i=0}^{3}(\curl^2\boldsymbol{v})(\texttt{v}_{i}).&
\end{align}
\end{subequations}

\begin{lemma}
\label{3d-w-dof-uni-solvent}
The DoFs \eqref{3d-wdof} are unisolvent for $W(T)$.
\end{lemma}
\begin{proof}
The number of the DoFs \eqref{3d-wdof} is
$$
6\times3+4\times(1+5+6)+3=69=\dim W(T).
$$

Take $ \boldsymbol{v}\in W(T)$ and assume all the DoFs \eqref{3d-wdof} vanish,
then we prove $\boldsymbol{v}=0$. 
For $F\in\mathcal{F}(T)$, $(\boldsymbol{v}\times\boldsymbol{n})|_F\in\mathbb P_2(F;\mathscr T^{F})+\mathbb P_2(F)\Pi_F\boldsymbol{x}$. By the unisolvence of the Raviart--Thomas element in two dimensions \cite{MR4458899,RaviartThomas1977,Nedelec1980}, the vanishing DoFs \eqref{3d-wdof1}--\eqref{3d-wdof3} imply that $(\boldsymbol{v}\times\boldsymbol{n})|_F=0$. Then $\boldsymbol{v}\in H_0(\curl,T)$ and $\curl\boldsymbol{v}\in H_0(\div,T)$.
By the vanishing DoFs \eqref{3d-wdof4}--\eqref{3d-wdof5}, and Lemma~\ref{dof-uni-solvent}, we obtain $\curl\boldsymbol{v}=0$.
So $\boldsymbol{v}=\nabla w$ with $w\in\mathbb P_3(T)$ satisfying $w(\texttt{v}_0)=0$.
By $\boldsymbol{v}\in H_0(\curl,T)$, we conclude $w\in H_0^1(T)$.
Therefore, $w=0$ and $\boldsymbol{v}=\nabla w=0$.
\end{proof}

Define the global $H(\grad\curl)$-nonconforming finite element space
\begin{align*}
W_{h}&=\{\boldsymbol{v}_{h}\in L^{2}(\Omega;\mathbb{R}^{3}):\boldsymbol{v}_h|_T\in W(T)~\textrm{for}~T\in\mathcal{T}_h; \textrm{all~the~DoFs \eqref{3d-wdof} are} \\
&\qquad\qquad\;\;\textrm{single-valued, and DoFs \eqref{3d-wdof1}--\eqref{3d-wdof4} on boundary vanish}\}.
\end{align*}
We have $W_{h}\subset H_0(\curl,\Omega)$, but $W_{h}\not\subseteq H_0(\grad\curl,\Omega)$.

Next, we present the nonconforming discretization of the Stokes complex \eqref{Stokescomplex3d}.
\begin{lemma}
The nonconforming discrete complex
\begin{align}
\label{3dcomplex-2}
0\xrightarrow{\subset}\mathring{\mathbb{P}}^c_3(\mathcal{T}_h)\xrightarrow{\nabla} W_h \xrightarrow{\curl}  V_h
\xrightarrow{\div} \mathring{\mathbb{P}}_1(\mathcal{T}_h) \xrightarrow {}0
\end{align}
is exact.
\end{lemma}
\begin{proof}
First, the sequence \eqref{3dcomplex-2} is a complex.
By \eqref{eq:div-stability}, $\div V_h=\mathring{\mathbb{P}}_1(\mathcal{T}_h)$, which means
\begin{equation*}
\dim V_h\cap\ker(\div)=\dim V_h -\dim \mathring{\mathbb{P}}_1(\mathcal{T}_h)=12|\mathring{\mathcal{F}}_h|+2|\mathcal{T}_h|+1.
\end{equation*}

Next we verify $W_h\cap\ker(\curl)=\nabla \mathring{\mathbb{P}}^c_3(\mathcal{T}_h)$.
For $\boldsymbol{w}_h\in W_h\cap\ker(\curl)\subset H_0(\curl,\Omega)\cap\ker(\curl)$,
it follows that $\boldsymbol{w}_h=\nabla v_h$ with $v_h\in H_0^1(\Omega)$ satisfying $v_h|_T\in \mathbb{P}_3(T)$ for $T\in\mathcal T_h$. Then $v_h\in \mathring{\mathbb{P}}^c_3(\mathcal{T}_h)$, and $\boldsymbol{w}_h\in\nabla \mathring{\mathbb{P}}^c_3(\mathcal{T}_h)$.
As a result,
\begin{equation*}
\dim\curl W_h=\dim W_h-\dim \mathring{\mathbb{P}}^c_3(\mathcal{T}_h)= |\mathring{\mathcal{E}}_h|+11|\mathring{\mathcal{F}}_h|+3|\mathcal{T}_h|-|\mathring{\mathcal{V}}_h|.
\end{equation*} 

By Euler's formula, we obtain
\begin{equation*}
\dim V_h\cap\ker(\div) - \dim\curl W_h= |\mathring{\mathcal{V}}_h|-|\mathring{\mathcal{E}}_h| +|\mathring{\mathcal{F}}_h|-|\mathcal{T}_h|+1=0,
\end{equation*}
which implies $V_h\cap\ker(\div)=\curl W_h$.
\end{proof}

Let $I_h^{SZ}: H_0^1(\Omega)\to \mathring{\mathbb{P}}^c_3(\mathcal{T}_h)$ be the Scott--Zhang interpolation operator \cite{ScottZhang1990}.
We follow the idea in \cite[Section 4.2]{MR4621133} to construct an interpolation operator $I_h^W: H_0(\grad\curl,\Omega)\to W_h$. Define the interpolation operator $\widetilde{I}_h^W: H_0^2(\Omega;\mathbb R^3)\to W_h$ as follows: 
\begin{align*}
\int_{e}(\widetilde{I}_h^W\boldsymbol{v})\cdot\boldsymbol{t}\,q \ds &= \int_{e}\boldsymbol{v} \cdot\boldsymbol{t}\,q \ds, \qquad\qquad\quad\;\, q\in\mathbb{P}_2(e), e\in\mathring{\mathcal{E}}_h,\\
\int_{F}(\Pi_F(\widetilde{I}_h^W\boldsymbol{v})) \cdot \boldsymbol{q} \ds &= \int_{F}(\Pi_F\boldsymbol{v}) \cdot \boldsymbol{q} \ds, \qquad\quad\;\;\, \boldsymbol{q}\in\boldsymbol{x}\mathbb{P}_0(F), F\in\mathring{\mathcal{F}}_h,\\
\int_{F}\curl(\widetilde{I}_h^W\boldsymbol{v}) \cdot\boldsymbol{n} \, q \ds &= \int_{F}(\curl\boldsymbol{v}) \cdot\boldsymbol{n} \, q \ds,  \quad \;\; q\in\mathbb{P}_2(F)/\mathbb{R}, F\in\mathring{\mathcal{F}}_h,\\
\int_{F}\Pi_F\curl(\widetilde{I}_h^W\boldsymbol{v}) \cdot \boldsymbol{q} \ds &= \int_{F}\Pi_F(\curl\boldsymbol{v}) \cdot \boldsymbol{q} \ds,  \;\;\;\; \boldsymbol{q}\in\mathbb{P}_1(F;\mathscr T_F), F\in\mathring{\mathcal{F}}_h,\\
\sum_{i=0}^{3}(\curl^2(\widetilde{I}_h^W\boldsymbol{v}))(\texttt{v}_{i}) &=\sum_{i=0}^{3}(\curl I_h(\curl\boldsymbol{v}))(\texttt{v}_{i}),\quad T\in\mathcal{T}_h.
\end{align*}
It can be verified that
\begin{equation}\label{commutative3dcurl0}
\curl(\widetilde{I}_h^W\boldsymbol{v})=I_h(\curl\boldsymbol{v})\quad\forall~\boldsymbol{v}\in H_0^2(\Omega;\mathbb R^3).
\end{equation}

Applying Lemma 4.5 in \cite{MR4621133}, we have the regular decomposition
\begin{equation}\label{Wregulardecomp}
H_0(\grad\curl,\Omega) = H_0^2(\Omega;\mathbb R^3) + \nabla H_0^1(\Omega).
\end{equation}
For $\boldsymbol{v}\in H_0(\grad\curl,\Omega)$, by the regularity decomposition \eqref{Wregulardecomp}, write $\boldsymbol{v}=\boldsymbol{v}_2+\nabla v_1$ with $\boldsymbol{v}_2\in H_0^2(\Omega;\mathbb R^3)$ and $v_1\in H_0^1(\Omega)$. Here $\boldsymbol{v}_2$ satisfies equation (4.19) in \cite{MR4621133}.
When $\boldsymbol{v}\in\nabla H_0^1(\Omega)$, $\boldsymbol{v}_2=0$.
Define the interpolation 
$$
I_h^W\boldsymbol{v}=\widetilde{I}_h^W\boldsymbol{v}_2+\nabla(I_h^{SZ}v_1).
$$
By the definition of $I_h^W$,
\begin{equation}\label{commutative3dgrad}
I_h^W(\nabla v) = \nabla(I_h^{SZ}v)\quad\forall~v\in H_0^1(\Omega).
\end{equation}
It follows from \eqref{commutative3dcurl0} that
\begin{equation}\label{commutative3dcurl}
\curl(I_h^W\boldsymbol{v})=\curl(\widetilde{I}_h^W\boldsymbol{v}_2)=I_h(\curl\boldsymbol{v})\quad\forall~\boldsymbol{v}\in H_0(\grad\curl,\Omega).
\end{equation}
The combination of complex \eqref{3dcomplex-2}, the commuting properties \eqref{commutative1} and \eqref{commutative3dgrad}--\eqref{commutative3dcurl} yields the following commutative diagram
$$
\xymatrix{
0 \ar[r]^{\subset\hspace{1.5em}} &
H_0^1(\Omega) \ar[r]^{\nabla\hspace{2.2em}} \ar[d]^{I^{SZ}_h} &
H_0(\grad\curl,\Omega) \ar[r]^{\hspace{1.2em}\curl} \ar[d]^{I^{W}_h} &
H_0^1(\Omega;\mathbb{R}^3) \ar[r]^{\hspace{1em}\div} \ar[d]^{I_h} &
L_0^2(\Omega) \ar[r] \ar[d]^{Q_{1,h}} & 0\\
0 \ar[r]^{\subset\hspace{1.5em}} & \mathring{\mathbb{P}}^c_3(\mathcal{T}_h) \ar[r]^{\nabla} &
W_h \ar[r]^{\curl} &  V_h \ar[r]^{\div} & \mathring{\mathbb{P}}_1(\mathcal{T}_h) \ar[r] & 0. 
}
$$

\section{Interior Penalty Nonconforming Finite Element Method}\label{sec4}
In this section, we propose an interior penalty nonconforming FEM for \eqref{weak1} and derive optimal and robust error estimates. 
\subsection{Discrete formulation}
By applying the interior penalty technique in~\cite{douglasdupont1976,Baker1977,wheeler1978},
we propose the following IP nonconforming FEM for the weak formulation \eqref{weak1}: find $\boldsymbol{u}_h\in V_h$ such that
\begin{align}\label{IPDG}
\iota^2 a_h(\boldsymbol{u}_h,\boldsymbol{v}_h)+b_h(\boldsymbol{u}_h,\boldsymbol{v}_h)=(\boldsymbol{f},\boldsymbol{v}_h) \quad \forall \ \boldsymbol{v}_h\in V_h,
\end{align} 
where 
\begin{align*}
a_h(\boldsymbol{u}_h,\boldsymbol{v}_h)&:=(\nabla_h\boldsymbol{\sigma}_h(\boldsymbol{u}_h),\nabla_h\boldsymbol{\varepsilon}_h(\boldsymbol{v}_h))-\sum_{F\in\mathcal F_h}([\![\boldsymbol{\sigma}_h(\boldsymbol{u}_h)]\!],\big\{\partial_{n}\boldsymbol{\varepsilon}_h(\boldsymbol{v}_h)\big\})_F\\
&\!\!\!\!\!\!-\sum_{F\in\mathcal F_h}(\big\{\partial_{n}\boldsymbol{\sigma}_h(\boldsymbol{u}_h)\big\},[\![\boldsymbol{\varepsilon}_h(\boldsymbol{v}_h)]\!])_F + \eta\sum_{F\in\mathcal F_h}h^{-1}_F([\![\boldsymbol{\sigma}_h(\boldsymbol{u}_h)]\!],[\![\boldsymbol{\varepsilon}_h(\boldsymbol{v}_h)]\!])_F,\\
b_h(\boldsymbol{u}_h,\boldsymbol{v}_h)&:=(\boldsymbol{\sigma}_h(\boldsymbol{u}_h), \boldsymbol{\varepsilon}_h(\boldsymbol{v}_h)).
\end{align*}
Here, the penalty parameter $\eta>0$ is a positive constant.


Thanks to the weak continuity \eqref{weak-continuity-vh1},  we can employ the proof of Lemma 4.1 in~\cite{HuangHuangTang2024} to derive the following Korn's inequalities.
\begin{lemma}
We have the discrete Korn's inequality
\begin{equation}\label{discreteKorn}
\|\boldsymbol{v}\|_{1,h}\lesssim \| \boldsymbol{\varepsilon}_h(\boldsymbol{v})\|_{0} \quad \forall~\boldsymbol{v}\in  V_h,
\end{equation}
and the discrete $H^2$-Korn's inequality
\begin{equation}\label{discreteKornH2}
\|\boldsymbol{v}\|_{2,h}\lesssim \interleave \boldsymbol{\varepsilon}_h(\boldsymbol{v})\interleave_{1,h}  \quad \forall~\boldsymbol{v}\in  V_h.
\end{equation}
\end{lemma}
\begin{proof}
By (1.22) in \cite{MR2047078} and (3.8) in \cite{ChenHuHuang2018},
the discrete Korn's inequality \eqref{discreteKorn} follows from the weak continuity \eqref{weak-continuity-vh1}.

Using the inequality $(1-1/\sqrt{2})\|\nabla^2\boldsymbol{v}\|^2_{0,T}\leq\|\nabla \boldsymbol{\varepsilon}(\boldsymbol{v})\|^2_{0,T}$ (cf. \cite[(7)]{MR4296093}) together with the discrete Korn's inequality \eqref{discreteKorn}, we obtain
\begin{equation*}
\|\boldsymbol{v}\|_{2,h}\lesssim |\boldsymbol{\varepsilon}_h(\boldsymbol{v})|_{1,h}+\| \boldsymbol{\varepsilon}_h(\boldsymbol{v})\|_{0} \quad \forall~\boldsymbol{v}\in  V_h.
\end{equation*}
Moreover, employing the discrete Poincar\'e inequality in \cite[Remark 1.1]{MR1974504} yields
\begin{equation*}
\| \boldsymbol{\varepsilon}_h(\boldsymbol{v})\|_{0}\lesssim \interleave \boldsymbol{\varepsilon}_h(\boldsymbol{v})\interleave_{1,h} \quad\forall~\boldsymbol{v}\in V_h.
\end{equation*}
Combining the last two inequalities gives \eqref{discreteKornH2}.
\end{proof}

\begin{theorem}
\label{wellposed-lemma}
There exists a constant $\eta_0>0$, depending only on the dimension
and the shape regularity of $\mathcal T_h$, such that, for every
$\eta\geq\eta_0$, the nonconforming finite element method
\eqref{IPDG} is well posed. In particular, $\eta_0$ is
independent of $h$, $\iota$, and $\lambda$.
\end{theorem}
\begin{proof}
For any $\eta\geq\eta_0$, the standard coercivity estimate for
symmetric interior penalty methods gives
(cf. \cite{MozolevskiBoesing2007,EpshteynRiviere2007})
\begin{equation*}
2\mu{\interleave \boldsymbol{\varepsilon}_h(\boldsymbol{v})\interleave}_{1,h}^2 + \lambda\interleave\!\div\boldsymbol{v}\interleave^2_{1,h} \lesssim a_{h}(\boldsymbol{v}, \boldsymbol{v})  \quad \forall~\boldsymbol{v}\in  V_h.
\end{equation*}
Then
\begin{equation}\label{eq:abhelliptic}
\interleave \boldsymbol{v}\interleave^2_{\iota,\lambda,h}\eqsim \iota^2a_h(\boldsymbol{v}, \boldsymbol{v})+b_h(\boldsymbol{v}, \boldsymbol{v}) \quad \forall~\boldsymbol{v}\in  V_h.
\end{equation}
Hence, the well-posedness of the nonconforming finite element method \eqref{IPDG} follows from the Lax--Milgram lemma \cite{Ciarlet1978}.
\end{proof}

\subsection{Error analysis}
To present the error analysis, we first construct an $H^1$-conforming virtual element space following the idea in \cite{HuangWang2023}.
The local virtual element space of shape functions is defined by
\begin{align*}
V^{\rm VE}(T)&:=\{\boldsymbol{v}\in H^1(T;\mathbb{R}^d):\boldsymbol{v}|_{F}\in \mathbb{P}_{2d+1}(F;\mathbb{R}^d)~\textrm{for}~F\in\mathcal{F}(T),~\div\boldsymbol{v}\in\mathbb{P}_1(T),\\
&\quad\quad\;\textrm{there~exists~some}~p\in L^2(T)~\textrm{such~that}~\div\boldsymbol{\varepsilon}(\boldsymbol{v})+\nabla p \in \mathbb{G}_{2d-1}^{\oplus}(T)\},
\end{align*}
where
\begin{equation*}
\mathbb{G}_{2d-1}^{\oplus}(T):=\mathbb{P}_{2d-2}(T;\mathbb K)(\boldsymbol{x}-\boldsymbol{x}_T)=\{\boldsymbol{\tau}(\boldsymbol{x}-\boldsymbol{x}_T): \boldsymbol{\tau}\in \mathbb{P}_{2d-2}(T;\mathbb K)\}
\end{equation*}
with $\boldsymbol{x}_T$ being the barycenter of simplex $T$.
We can see that $V(T)\subseteq V^{\rm VE}(T)$.
The DoFs are given by
\begin{subequations}\label{vedof}
\begin{align}
	\label{vedof1}
\boldsymbol{v}(\texttt{v}),& \quad \ \texttt{v}\in\mathcal{V}(T),\\
\label{vedof2}
\int_{f}\boldsymbol{v}\cdot\boldsymbol{q} \ds,& \quad \ \boldsymbol{q}\in\mathbb{P}_{2d-\ell}(f;\mathbb R^d), f\in\Delta_{\ell}(T), \ell = 1,2,\ldots,d-1,\\
\label{vedof3} 
\int_{T}\boldsymbol{v}\cdot\boldsymbol{q} \dx,& \quad \ \boldsymbol{q}\in\mathbb{P}_0(T;\mathbb{R}^d)\oplus \mathbb{G}_{2d-1}^{\oplus}(T),
\end{align}
where $\Delta_{\ell}(T)$ denotes the set of all $\ell$-dimensional subsimplices of $T$.
\end{subequations}
Following the argument in \cite[Section 2.4]{HuangWang2023}, we can show that the DoFs \eqref{vedof} are unisolvent for space $V^{\rm VE}(T)$.

Introduce the global virtual element space as
\begin{equation*}
V^{\rm VE}_h:=\{\boldsymbol{v}\in H^1_0(\Omega;\mathbb{R}^d):\boldsymbol{v}|_T\in V^{\rm VE}(T)~\textrm{for~each}~T\in\mathcal{T}_h\}.
\end{equation*}
Define a connection operator $E_h: V_h+H_0^1(\Omega;\mathbb{R}^d)\rightarrow V^{\rm VE}_h$ as follows: for any $\boldsymbol{v}\in V_h+H_0^1(\Omega;\mathbb{R}^d)$,  $E_h\boldsymbol{v}\in  V^{\rm VE}_h$ is determined by
\begin{align}
\notag	
(E_h\boldsymbol{v})(\texttt{v})&=\frac{1}{|\mathcal{T}_{\texttt{v}}|}\sum_{T\in\mathcal{T}_{\texttt{v}}}(Q_{2,T}\boldsymbol{v})(\texttt{v}), \qquad\quad\;\;\; \ \texttt{v}\in\mathring{\mathcal{V}}_h,\\
\notag
\int_{f}(E_h\boldsymbol{v})\cdot\boldsymbol{q} \ds&=\frac{1}{|\mathcal{T}_{f}|}\sum_{T\in\mathcal{T}_{f}}\int_{f}(Q_{2,T}\boldsymbol{v})\cdot\boldsymbol{q} \ds, \quad \ \boldsymbol{q}\in\mathbb{P}_{2d-\ell}(f;\mathbb R^d),  \\
\notag
&\qquad\qquad\qquad\qquad\qquad\qquad\qquad f\in\Delta_{\ell}(\mathring{\mathcal{T}}_h), \ell = 1,\ldots,d-2,\\
\label{connection-def0}
\int_{F}(E_h\boldsymbol{v})\cdot\boldsymbol{q} \ds&=\frac{1}{|\mathcal{T}_{F}|}\sum_{T\in\mathcal{T}_{F}}\int_{F}\boldsymbol{v}|_T\cdot\boldsymbol{q} \ds, \quad\;\;\;\; \ \boldsymbol{q}\in\mathbb{P}_{d+1}(F;\mathbb R^{d}), F\in\mathring{\mathcal{F}}_h,\\
\notag
\int_{T}(E_h\boldsymbol{v})\cdot\boldsymbol{q} \dx&=\int_{T}\boldsymbol{v}\cdot\boldsymbol{q} \dx, \quad \qquad\qquad\qquad\;\;\ \boldsymbol{q}\in\mathbb{P}_0(T;\mathbb{R}^d)\oplus \mathbb{G}_{2d-1}^{\oplus}(T),
\end{align}
where $\mathcal{T}_f$ is the set of all simplices in $\mathcal{T}_h$ sharing the common $f$, and $\Delta_{\ell}(\mathring{\mathcal{T}}_h)$ denotes the set of all interior $\ell$-dimensional subsimplices of $\mathcal{T}_h$.
By $V_h+H_0^1(\Omega;\mathbb{R}^d)\subseteq H_0(\div,\Omega)$,
equation \eqref{connection-def0} indicates
\begin{equation*}
\int_{F}(E_h\boldsymbol{v})\cdot\boldsymbol{n}\,q \ds= \int_{F}\boldsymbol{v}\cdot\boldsymbol{n}\,q \ds,  \quad \ q\in\mathbb{P}_{d+1}(F), F\in\mathring{\mathcal{F}}_h.
\end{equation*}
This together with the integration by parts yields
\begin{equation}
\label{connection-change}
\div(E_h\boldsymbol{v})=Q_{1,h}(\div\boldsymbol{v}) \quad \forall \ \boldsymbol{v}\in V_h+H_0^1(\Omega;\mathbb{R}^d).
\end{equation}
Applying arguments similar to those used in the proofs of Lemmas 5.1--5.3 in \cite{HuangWang2023}, together with the weak continuity \eqref{weak-continuity-vh1} of $ V_h$, we obtain
\begin{align}
\label{connection-Vh1}
|\boldsymbol{v}-E_h\boldsymbol{v}|_{m,h}\lesssim h^{s-m}|\boldsymbol{v}|_{s,h}\quad \forall \ \boldsymbol{v}\in V_h+H^2_0(\Omega;\mathbb{R}^d), 0\leq m\leq 1\leq s\leq 2.
\end{align}
By \eqref{connection-change},
\begin{align}
\label{connection-Vh2}
\|\div(\boldsymbol{v}-E_h\boldsymbol{v})\|_{0}\lesssim h^{s}|\div\boldsymbol{v}|_{s,h}\quad \forall \ \boldsymbol{v}\in V_h+H^2_0(\Omega;\mathbb{R}^d), 0\leq s\leq 1.
\end{align}
\begin{lemma}
\label{interpolation-I-lemma7}
Let $\boldsymbol{u}\in H_0^2(\Omega;\mathbb{R}^d)\cap H^3(\Omega;\mathbb{R}^d)$ be the solution of problem \eqref{SGE0}, and the solution $\boldsymbol{u}_0\in H^1_0(\Omega;\mathbb{R}^d)$ of problem \eqref{SGElinear} satisfy the regularity \eqref{elasregularity} with $2\leq s\leq 3$. We have for any $\boldsymbol{v}_h\in V_h$ that
\begin{align}
\label{Ih-error4}
&\quad\;\iota^2a_h(\boldsymbol{u}-I_h\boldsymbol{u},\boldsymbol{v}_h)+b_h(\boldsymbol{u}-I_h\boldsymbol{u},\boldsymbol{v}_h) \\
\notag
&\qquad\qquad\qquad\qquad\qquad\lesssim \iota^{-1/2}h\|\boldsymbol{f}\|_0(\|\boldsymbol{\varepsilon}_h(\boldsymbol{v}_h)\|_{0}+\iota\interleave \boldsymbol{\varepsilon}_h(\boldsymbol{v}_h)\interleave_{1,h}), \\
\label{Ih-error5}
&\quad\;\iota^2a_h(\boldsymbol{u}-I_h\boldsymbol{u},\boldsymbol{v}_h)+b_h(\boldsymbol{u}-I_h\boldsymbol{u},\boldsymbol{v}_h) \\
\notag
&\qquad\qquad\qquad\qquad\qquad\lesssim(\iota^{1/2}\|\boldsymbol{f}\|_0+h^{s-1}\|\boldsymbol{f}\|_{s-2})\|\boldsymbol{\varepsilon}_h(\boldsymbol{v}_h)\|_0.
\end{align}
\end{lemma}
\begin{proof}
Employing the Cauchy--Schwarz inequality, the inverse inequality, \eqref{Ih-error1}--\eqref{Ih-error11},  \eqref{Jh-error1}--\eqref{Jh-error11} and \eqref{interpolation-I-lemma2-pf},
\begin{align*}
a_h(\boldsymbol{u}-I_h\boldsymbol{u},\boldsymbol{v}_h)
&\lesssim \min\{h|\boldsymbol{\sigma}(\boldsymbol{u})|_2\interleave\boldsymbol{\varepsilon}_h(\boldsymbol{v}_h)\interleave_{1,h},|\boldsymbol{\sigma}(\boldsymbol{u})|_2\|\boldsymbol{\varepsilon}_h(\boldsymbol{v}_h)\|_{0}\}, 
\end{align*}
and
\begin{align*}
&b_h(\boldsymbol{u}-I_h\boldsymbol{u},\boldsymbol{v}_h)\leq\|\boldsymbol{\sigma}_h(\boldsymbol{u}-I_h\boldsymbol{u})\|_0\|\boldsymbol{\varepsilon}_h(\boldsymbol{v}_h)\|_0\\
\lesssim&\min\{h|\boldsymbol{\sigma}(\boldsymbol{u})|_1,\|\boldsymbol{\sigma}(\boldsymbol{u}-\boldsymbol{u}_0)\|_{0}+h^{s-1}|\boldsymbol{u}_0|_s+\lambda h^{s-1}|\div\boldsymbol{u}_0|_{s-1}\}\|\boldsymbol{\varepsilon}_h(\boldsymbol{v}_h)\|_0.
\end{align*}
Then estimates \eqref{Ih-error4}--\eqref{Ih-error5} follow from \eqref{elasregularity}--\eqref{Regularity-divu}.
\end{proof}

\begin{lemma}
\label{consistency-lemma1}
Let $\boldsymbol{u}\in H_0^2(\Omega;\mathbb{R}^d)\cap H^3(\Omega;\mathbb{R}^d)$ be the solution of problem \eqref{SGE0}. We have for any $\boldsymbol{v}_h\in V_h$ that
\begin{align}
\label{consistency-error30}
-a_h(\boldsymbol{u},\boldsymbol{v}_h)-(\Delta\boldsymbol{\sigma}(\boldsymbol{u}),\boldsymbol{\varepsilon}(E_h\boldsymbol{v}_h))&\lesssim \iota^{-3/2}h\|\boldsymbol{f}\|_0\interleave\boldsymbol{\varepsilon}_h(\boldsymbol{v}_h)\interleave_{1,h}, \\
\label{consistency-error3}
-a_h(\boldsymbol{u},\boldsymbol{v}_h)-(\Delta\boldsymbol{\sigma}(\boldsymbol{u}),\boldsymbol{\varepsilon}(E_h\boldsymbol{v}_h))&\lesssim \iota^{-3/2}\|\boldsymbol{f}\|_0\|\boldsymbol{\varepsilon}_h(\boldsymbol{v}_h)\|_{0}.
\end{align}
\end{lemma}
\begin{proof}
Applying the integration by parts, we have
\begin{align*}
-a_h(\boldsymbol{u},\boldsymbol{v}_h)-(\Delta\boldsymbol{\sigma}(\boldsymbol{u}),\boldsymbol{\varepsilon}(E_h\boldsymbol{v}_h))
=(\Delta\boldsymbol{\sigma}(\boldsymbol{u}),\boldsymbol{\varepsilon}_h(\boldsymbol{v}_h-E_h\boldsymbol{v}_h)).
\end{align*}
Hence, the estimates \eqref{consistency-error30}--\eqref{consistency-error3} follow from \eqref{connection-Vh1}, the regularity \eqref{Regularity-u}--\eqref{Regularity-divu} and the discrete Korn's inequalities \eqref{discreteKorn}--\eqref{discreteKornH2} immediately.
\end{proof}

\begin{lemma}
\label{consistency-lemma2}
Let $\boldsymbol{u}\in H_0^2(\Omega;\mathbb{R}^d)\cap H^3(\Omega;\mathbb{R}^d)$ be the solution of problem \eqref{SGE0}, and the solution $\boldsymbol{u}_0\in H^1_0(\Omega;\mathbb{R}^d)$ of problem \eqref{SGElinear} satisfy the regularity \eqref{elasregularity} with $2\leq s\leq 3$. We have for any $\boldsymbol{v}_h\in V_h$ that
\begin{align}
\label{main-proof-uh-20}
&(\boldsymbol{f},\boldsymbol{v}_h)-\iota^2 a_h(\boldsymbol{u},\boldsymbol{v}_h)-b_h(\boldsymbol{u},\boldsymbol{v}_h) \!\lesssim h\|\boldsymbol{f}\|_0(\|\boldsymbol{\varepsilon}_h(\boldsymbol{v}_h)\|_{0}+\iota^{1/2}\interleave \boldsymbol{\varepsilon}_h(\boldsymbol{v}_h)\interleave_{1,h}), \\
\label{main-proof-uh-2}
&(\boldsymbol{f},\boldsymbol{v}_h)-\iota^2 a_h(\boldsymbol{u},\boldsymbol{v}_h)-b_h(\boldsymbol{u},\boldsymbol{v}_h) \lesssim (\iota^{1/2}\|\boldsymbol{f}\|_0+h^{s-1}\|\boldsymbol{f}\|_{s-2})\|\boldsymbol{\varepsilon}_h(\boldsymbol{v}_h)\|_{0}.
\end{align}
\end{lemma}
\begin{proof}
Let $\boldsymbol{w}_h=\boldsymbol{v}_h-E_h\boldsymbol{v}_h$ for simplicity. 
By applying the integration by parts to problem \eqref{SGElinear}, we have
\begin{equation}\label{eq:202408081}
(\boldsymbol{f}, \boldsymbol{w}_h)-b_h(\boldsymbol{u},\boldsymbol{w}_h)=(\boldsymbol{\sigma}(\boldsymbol{u}_0-\boldsymbol{u}),\boldsymbol{\varepsilon}_h(\boldsymbol{w}_h))-\sum_{F\in\mathcal{F}_h}(\boldsymbol{\sigma}(\boldsymbol{u}_0)\boldsymbol{n},[\![\boldsymbol{v}_h]\!])_F.
\end{equation}
Adopting \eqref{connection-Vh1}, it holds
\begin{equation}\label{eq:202408082}
(\boldsymbol{\sigma}(\boldsymbol{u}_0-\boldsymbol{u}),\boldsymbol{\varepsilon}_h(\boldsymbol{w}_h))
\lesssim \|\boldsymbol{\sigma}(\boldsymbol{u}_0-\boldsymbol{u})\|_0\min\{|\boldsymbol{v}_h|_{1,h}, h|\boldsymbol{v}_h|_{2,h}\}.
\end{equation}
Thanks to the weak continuity \eqref{weak-continuity-vh1} of $V_{h}$ and $V_{h}\subset H_0(\div,\Omega)$, we obtain
\begin{align*}
-\sum_{F\in\mathcal{F}_h}(\boldsymbol{\sigma}(\boldsymbol{u}_0)\boldsymbol{n},[\![\boldsymbol{v}_h]\!])_F&=-2\mu\sum_{F\in\mathcal{F}_h}(\boldsymbol{\varepsilon}(\boldsymbol{u}_0)\boldsymbol{n},[\![\boldsymbol{v}_h]\!])_F\\
&=-2\mu\sum_{F\in\mathcal{F}_h}(\boldsymbol{\varepsilon}(\boldsymbol{u}_0)\boldsymbol{n}-Q_{1,F}(\boldsymbol{\varepsilon}(\boldsymbol{u}_0)\boldsymbol{n}),[\![\boldsymbol{v}_h]\!]-Q_{1,F}[\![\boldsymbol{v}_h]\!])_F.
\end{align*}
From the error estimate of $Q_{1,F}$, the trace inequality and \eqref{elasregularity}, we acquire
\begin{equation}
\label{consistency-error4}
-\sum_{F\in\mathcal{F}_h}(\boldsymbol{\sigma}(\boldsymbol{u}_0)\boldsymbol{n},[\![\boldsymbol{v}_h]\!])_F\lesssim h^{s-1}\|\boldsymbol{f}\|_{s-2}|\boldsymbol{v}_h|_{1,h},
\end{equation}
which together with  \eqref{eq:202408081}--\eqref{eq:202408082} and the regularity \eqref{Regularity-u}--\eqref{Regularity-divu} yields
\begin{align}
\label{consistency-error5}
(\boldsymbol{f},\boldsymbol{w}_h)-b_h(\boldsymbol{u},\boldsymbol{w}_h)&\lesssim \iota^{1/2}\|\boldsymbol{f}\|_0\min\{|\boldsymbol{v}_h|_{1,h}, h|\boldsymbol{v}_h|_{2,h}\} \\
\notag
&\quad+h^{s-1}\|\boldsymbol{f}\|_{s-2}|\boldsymbol{v}_h|_{1,h}.
\end{align}

On the other hand, apply the integration by parts to problem \eqref{SGE0} to get
\begin{align*}
(\boldsymbol{f},E_h\boldsymbol{v}_h)=-\iota^2(\Delta\boldsymbol{\sigma}(\boldsymbol{u}),\boldsymbol{\varepsilon}(E_h\boldsymbol{v}_h))+(\boldsymbol{\sigma}(\boldsymbol{u}),\boldsymbol{\varepsilon}(E_h\boldsymbol{v}_h)).
\end{align*}
Thus, we have
\begin{align*}
(\boldsymbol{f},\boldsymbol{v}_h)-\iota^2a_h(\boldsymbol{u},\boldsymbol{v}_h)-b_h(\boldsymbol{u},\boldsymbol{v}_h)
&=(\boldsymbol{f},\boldsymbol{w}_h)-b_h(\boldsymbol{u},\boldsymbol{w}_h) \\
&\quad\;-\iota^2\big(a_h(\boldsymbol{u},\boldsymbol{v}_h)+(\Delta\boldsymbol{\sigma}(\boldsymbol{u}),\boldsymbol{\varepsilon}(E_h\boldsymbol{v}_h))\big).
\end{align*}
Combining \eqref{consistency-error30}--\eqref{consistency-error3}, \eqref{consistency-error5}, and the discrete Korn's inequalities~\eqref{discreteKorn}--\eqref{discreteKornH2} yields \eqref{main-proof-uh-20}--\eqref{main-proof-uh-2}.
\end{proof}

\begin{theorem}
\label{main-theorem-uh-1}
Let $\boldsymbol{u}\in H_0^2(\Omega;\mathbb{R}^d)\cap H^3(\Omega;\mathbb{R}^d)$ be the solution of problem \eqref{SGE0}, and the solution $\boldsymbol{u}_0\in H^1_0(\Omega;\mathbb{R}^d)$ of problem \eqref{SGElinear} satisfy the regularity \eqref{elasregularity} with $2\leq s\leq 3$. Then
we have
\begin{align}
\label{main-result-uh-1}
\interleave\boldsymbol{u}-\boldsymbol{u}_h\interleave_{\iota,\lambda,h} &\lesssim \min\{\iota^{-1/2}h\|\boldsymbol{f}\|_0, \iota^{1/2}\|\boldsymbol{f}\|_0+h^{s-1}\|\boldsymbol{f}\|_{s-2}\}, \\
\label{main-result-uh-2}
\|\boldsymbol{u}_0-\boldsymbol{u}_{h}\|_{\iota,\lambda,h}&\lesssim \iota^{1/2}\|\boldsymbol{f}\|_0+h^{s-1}\|\boldsymbol{f}\|_{s-2}. 
\end{align}
\end{theorem}
\begin{proof}
Let $\boldsymbol{w}_h=\boldsymbol{u}_h-I_h\boldsymbol{u}$ for simplicity. By \eqref{eq:abhelliptic} and \eqref{IPDG}, it holds
\begin{align*}
\interleave\boldsymbol{u}_h-I_h\boldsymbol{u}\interleave^2_{\iota,\lambda,h} &\lesssim\iota^2a_h(\boldsymbol{u}_h-I_h\boldsymbol{u},\boldsymbol{w}_h)+b_h(\boldsymbol{u}_h-I_h\boldsymbol{u},\boldsymbol{w}_h) \\
& = (\boldsymbol{f},\boldsymbol{w}_h)- \iota^2a_h(I_h\boldsymbol{u},\boldsymbol{w}_h)-b_h(I_h\boldsymbol{u},\boldsymbol{w}_h)\\
&=(\boldsymbol{f},\boldsymbol{w}_h)-\iota^2a_h(\boldsymbol{u},\boldsymbol{w}_h)-b_h(\boldsymbol{u},\boldsymbol{w}_h) \\
&\quad\;+\iota^2a_h(\boldsymbol{u}-I_h\boldsymbol{u},\boldsymbol{w}_h)+b_h(\boldsymbol{u}-I_h\boldsymbol{u},\boldsymbol{w}_h).
\end{align*}
Then we get from \eqref{Ih-error4}--\eqref{Ih-error5} and \eqref{main-proof-uh-20}--\eqref{main-proof-uh-2} that
\begin{equation*}
\interleave\boldsymbol{u}_h-I_h\boldsymbol{u}\interleave_{\iota,\lambda,h} \lesssim \min\{\iota^{-1/2}h\|\boldsymbol{f}\|_0, \iota^{1/2}\|\boldsymbol{f}\|_0+h^{s-1}\|\boldsymbol{f}\|_{s-2}\},
\end{equation*}
which together with \eqref{Regularity-u}--\eqref{Regularity-divu} and \eqref{Ih-error2}--\eqref{Ih-error3}
gives \eqref{main-result-uh-1}.

Finally, we conclude \eqref{main-result-uh-2} from \eqref{main-result-uh-1} and \eqref{elasregularity}--\eqref{Regularity-divu}.
\end{proof}

It follows from \eqref{main-result-uh-1} that $$\|\boldsymbol{\varepsilon}_h(\boldsymbol{u}-\boldsymbol{u}_h)\|_{0}+\sqrt{\lambda}\|\div(\boldsymbol{u}-\boldsymbol{u}_h)\|_0\lesssim \iota^{1/2}\|\boldsymbol{f}\|_0+h^{s-1}\|\boldsymbol{f}\|_{s-2}.$$
In the following, we consider the $H^1$ error estimate when $\iota\eqsim 1$.
\begin{theorem}
Suppose Assumptions \ref{assumption1}--\ref{assumption2} hold and $\iota\eqsim 1$. Let $\boldsymbol{u}\in H_0^2(\Omega;\mathbb{R}^d)\cap H^3(\Omega;\mathbb{R}^d)$ be the solution of problem \eqref{SGE0}, and the solution $\boldsymbol{u}_0\in H^1_0(\Omega;\mathbb{R}^d)$ of problem \eqref{SGElinear} satisfy the regularity \eqref{elasregularity} with $2\leq s\leq 3$. Then we have
\begin{equation}
\label{H1-result}
\|\boldsymbol{\varepsilon}_h(\boldsymbol{u}-\boldsymbol{u}_h)\|_{0}+\sqrt{\lambda}\|\div(\boldsymbol{u}-\boldsymbol{u}_h)\|_0\lesssim\iota^{-3/2}h^2\|\boldsymbol{f}\|_0.
\end{equation}
\end{theorem}
\begin{proof}
Consider the dual problem: find $\boldsymbol{\widetilde{u}}\in H_0^2(\Omega;\mathbb{R}^d)$ such that
\begin{equation}
\label{dual-u}
\iota^2\div\Delta\boldsymbol{\sigma}(\boldsymbol{\widetilde{u}})-\div\boldsymbol{\sigma}(\boldsymbol{\widetilde{u}})=-\div\boldsymbol{\sigma}(E_h(\boldsymbol{u}-\boldsymbol{u}_h)).
\end{equation}
Multiplying \eqref{dual-u} by $\boldsymbol{\widetilde{u}}$, it holds from the integration by parts that
\begin{align*}
&\quad\iota^2(2\mu|\boldsymbol{\varepsilon}(\boldsymbol{\widetilde{u}})|^2_1+\lambda|\div\boldsymbol{\widetilde{u}}|^2_1)+2\mu\|\boldsymbol{\varepsilon}(\boldsymbol{\widetilde{u}})\|^2_0+\lambda\|\div\boldsymbol{\widetilde{u}}\|^2_0\\
&\leq \|\boldsymbol{\varepsilon}(E_h(\boldsymbol{u}-\boldsymbol{u}_h))\|_{0}\|\boldsymbol{\varepsilon}(\boldsymbol{\widetilde{u}})\|_0+ \lambda\|\div(E_h(\boldsymbol{u}-\boldsymbol{u}_h))\|_0\|\div\boldsymbol{\widetilde{u}}\|_0,
\end{align*}
which implies
\begin{align*}
&\quad\iota(|\boldsymbol{\widetilde{u}}|_2+\sqrt{\lambda}|\div\boldsymbol{\widetilde{u}}|_1)+\|\boldsymbol{\widetilde{u}}\|_1+\sqrt{\lambda}\|\div\boldsymbol{\widetilde{u}}\|_0\\
&\lesssim\|\boldsymbol{\varepsilon}(E_h(\boldsymbol{u}-\boldsymbol{u}_h))\|_{0}+\sqrt{\lambda}\|\div(E_h(\boldsymbol{u}-\boldsymbol{u}_h))\|_{0}.
\end{align*}
By \eqref{fourregularity1}--\eqref{fourregularity2}, 
\begin{align*}
\iota^{2}\|\div\boldsymbol{\widetilde{u}}\|_2&\lesssim\frac{1}{\lambda}(\|\boldsymbol{\varepsilon}(E_h(\boldsymbol{u}-\boldsymbol{u}_h))\|_{0}+\|\boldsymbol{\widetilde{u}}\|_1)+ \|\div(E_h(\boldsymbol{u}-\boldsymbol{u}_h))\|_{0}+\|\div\boldsymbol{\widetilde{u}}\|_0, \\
\iota^{2}\|\boldsymbol{\widetilde{u}}\|_3&\lesssim(1+\dfrac{1}{\lambda})(\|\boldsymbol{\varepsilon}(E_h(\boldsymbol{u}-\boldsymbol{u}_h))\|_{0}+\|\boldsymbol{\widetilde{u}}\|_1)+ \|\div(E_h(\boldsymbol{u}-\boldsymbol{u}_h))\|_{0}.
\end{align*}
Combining the last three equations gives
\begin{equation}
\label{dual-regularity-u}
\sum_{i=0}^2\iota^i(\|\boldsymbol{\widetilde{u}}\|_{i+1}+\sqrt{\lambda}\|\div\boldsymbol{\widetilde{u}}\|_i) 
\lesssim\|\boldsymbol{\varepsilon}(E_h(\boldsymbol{u}-\boldsymbol{u}_h))\|_{0}+\sqrt{\lambda}\|\div(E_h(\boldsymbol{u}-\boldsymbol{u}_h))\|_{0}.
\end{equation}
Multiply \eqref{dual-u} by $E_h(\boldsymbol{u}-\boldsymbol{u}_h)$ and apply the integration by parts to get
\begin{equation}
\label{H1-result-u-pf1}
\begin{aligned}
&\quad2\mu\|\boldsymbol{\varepsilon}(E_h(\boldsymbol{u}-\boldsymbol{u}_h))\|^2_{0}+\lambda\|\div(E_h(\boldsymbol{u}-\boldsymbol{u}_h))\|^2_{0}\\
&= -\iota^2(\Delta\boldsymbol{\sigma}(\boldsymbol{\widetilde{u}}),\boldsymbol{\varepsilon}(E_h(\boldsymbol{u}-\boldsymbol{u}_h)))+(\boldsymbol{\sigma}(\boldsymbol{\widetilde{u}}),\boldsymbol{\varepsilon}(E_h(\boldsymbol{u}-\boldsymbol{u}_h)))\\
&=\iota^2a_h(\boldsymbol{u}-\boldsymbol{u}_h,\boldsymbol{\widetilde{u}})+b_h(\boldsymbol{u}-\boldsymbol{u}_h,\boldsymbol{\widetilde{u}})\\
&\quad+\iota^2(\Delta\boldsymbol{\sigma}(\boldsymbol{\widetilde{u}}),\boldsymbol{\varepsilon}_h(\boldsymbol{u}-\boldsymbol{u}_h-E_h(\boldsymbol{u}-\boldsymbol{u}_h)))\\
&\quad-(\boldsymbol{\sigma}(\boldsymbol{\widetilde{u}}),\boldsymbol{\varepsilon}_h(\boldsymbol{u}-\boldsymbol{u}_h-E_h(\boldsymbol{u}-\boldsymbol{u}_h))) \\
&= I_1+I_2,
\end{aligned}
\end{equation}
where
\begin{align*}
I_1&:=\iota^2a_h(\boldsymbol{u}-\boldsymbol{u}_h,\boldsymbol{\widetilde{u}}-I_h\boldsymbol{\widetilde{u}})+b_h(\boldsymbol{u}-\boldsymbol{u}_h,\boldsymbol{\widetilde{u}}-I_h\boldsymbol{\widetilde{u}})\\
&\quad\;+\iota^2(\Delta\boldsymbol{\sigma}(\boldsymbol{\widetilde{u}}),\boldsymbol{\varepsilon}_h(\boldsymbol{u}-\boldsymbol{u}_h-E_h(\boldsymbol{u}-\boldsymbol{u}_h))) \\
&\quad\; -(\boldsymbol{\sigma}(\boldsymbol{\widetilde{u}}),\boldsymbol{\varepsilon}_h(\boldsymbol{u}-\boldsymbol{u}_h-E_h(\boldsymbol{u}-\boldsymbol{u}_h))), \\
I_2&:=\iota^2a_h(\boldsymbol{u}-\boldsymbol{u}_h,I_h\boldsymbol{\widetilde{u}})+b_h(\boldsymbol{u}-\boldsymbol{u}_h,I_h\boldsymbol{\widetilde{u}}).
\end{align*}
By the Cauchy--Schwarz inequality, \eqref{Ih-error2}, \eqref{connection-Vh1}--\eqref{connection-Vh2}, \eqref{dual-regularity-u} and \eqref{main-result-uh-1},
\begin{align}
\begin{split}
\label{H1-result-u-pf2}
I_1 &\lesssim {\interleave\boldsymbol{u}-\boldsymbol{u}_h\interleave_{\iota,\lambda,h}}{\interleave\boldsymbol{\widetilde{u}}-I_h\boldsymbol{\widetilde{u}}\interleave_{\iota,\lambda,h}} \\
&\quad+\iota^{-1}h{\interleave\boldsymbol{u}-\boldsymbol{u}_h\interleave_{\iota,\lambda,h}}\sum_{i=0}^2\iota^i(\|\boldsymbol{\widetilde{u}}\|_{i+1}+\sqrt{\lambda}\|\div\boldsymbol{\widetilde{u}}\|_i)\\
&\lesssim\iota^{-3/2}h^2\|\boldsymbol{f}\|_0(\|\boldsymbol{\varepsilon}(E_h(\boldsymbol{u}-\boldsymbol{u}_h))\|_{0}+\sqrt{\lambda}\|\div(E_h(\boldsymbol{u}-\boldsymbol{u}_h))\|_{0}).
\end{split}
\end{align}
We get from \eqref{IPDG}, \eqref{weak1} and the integration by parts that
\begin{align*}
I_2&=\iota^2a_h(\boldsymbol{u},I_h\boldsymbol{\widetilde{u}})+b_h(\boldsymbol{u},I_h\boldsymbol{\widetilde{u}})-(\boldsymbol{f},I_h\boldsymbol{\widetilde{u}})\\
&= \iota^2a_h(\boldsymbol{u},I_h\boldsymbol{\widetilde{u}}-\boldsymbol{\widetilde{u}})+b_h(\boldsymbol{u},I_h\boldsymbol{\widetilde{u}}-\boldsymbol{\widetilde{u}})-(\boldsymbol{f},I_h\boldsymbol{\widetilde{u}}-\boldsymbol{\widetilde{u}}) \\
&= -\iota^2(\Delta(\boldsymbol{\sigma}(\boldsymbol{u})), \boldsymbol{\varepsilon}_h(I_h\boldsymbol{\widetilde{u}}-\boldsymbol{\widetilde{u}}))+(\boldsymbol{\sigma}(\boldsymbol{u}), \boldsymbol{\varepsilon}_h(I_h\boldsymbol{\widetilde{u}}-\boldsymbol{\widetilde{u}}))-(\boldsymbol{f},I_h\boldsymbol{\widetilde{u}}-\boldsymbol{\widetilde{u}})\\
&\lesssim (\iota^2|\boldsymbol{\sigma}(\boldsymbol{u})|_2+\|\boldsymbol{\sigma}(\boldsymbol{u}-\boldsymbol{u}_0)\|_0)\|\boldsymbol{\varepsilon}_h(I_h\boldsymbol{\widetilde{u}}-\boldsymbol{\widetilde{u}})\|_0 + \sum_{T\in\mathcal{T}_h}(\boldsymbol{\sigma}(\boldsymbol{u}_0)\boldsymbol{n}, I_h\boldsymbol{\widetilde{u}}-\boldsymbol{\widetilde{u}})_{\partial T}.
\end{align*}
Then we acquire from the regularity results \eqref{Regularity-u}--\eqref{Regularity-divu}, \eqref{Ih-error1}, \eqref{consistency-error4} and \eqref{dual-regularity-u} that
\begin{align}
\label{H1-result-u-pf3}
I_2\lesssim \iota^{-3/2}h^2\|\boldsymbol{f}\|_0(\|\boldsymbol{\varepsilon}(E_h(\boldsymbol{u}-\boldsymbol{u}_h))\|_{0}+\sqrt{\lambda}\|\div(E_h(\boldsymbol{u}-\boldsymbol{u}_h))\|_{0}).
\end{align}
By \eqref{H1-result-u-pf1}--\eqref{H1-result-u-pf3}, it holds
\begin{equation*}
\|\boldsymbol{\varepsilon}_h(E_h(\boldsymbol{u}-\boldsymbol{u}_h))\|_{0}+\sqrt{\lambda}\|\div(E_h(\boldsymbol{u}-\boldsymbol{u}_h))\|_{0}\lesssim\iota^{-3/2}h^2\|\boldsymbol{f}\|_0,
\end{equation*}
which combined with the triangle inequality, \eqref{connection-Vh1}--\eqref{connection-Vh2} and \eqref{main-result-uh-1} gives \eqref{H1-result}. 
\end{proof}

\section{Numerical Experiments}\label{sec5}
In this section, we verify the theoretical convergence rates and investigate the robustness of the IP nonconforming FEM \eqref{IPDG} with respect to the size parameter $\iota$ and the Lam\'{e} coefficient $\lambda$ by means of numerical examples in both two and three dimensions. For the three-dimensional experiments, we use a hybridized implementation of \eqref{IPDG}, in which the interelement continuity constraints associated with the face DoFs are imposed by Lagrange multipliers.

In Subsections \ref{sec51}--\ref{sec52}, the tests are carried out on the unit domain $\Omega=(0,1)^d$ ($d=2,3$) using uniform simplicial meshes in both two and three dimensions. We set $\mu=1$, and choose the penalty parameter as $\eta=20$ in two dimensions and $\eta=40$ in three dimensions.

In Subsection \ref{sec53}, the IP nonconforming FEM~\eqref{IPDG} is further tested on a multiply connected domain through the uniaxial tension of a plate with a circular hole. The numerical results for several values of the size parameter are compared with the analytical Khakalo--Niiranen solution.

\subsection{Numerical results without boundary layers}\label{sec51}
We first consider exact solutions that are divergence-free and do not exhibit boundary layers. The right-hand side $\boldsymbol{f}$ is computed from \eqref{SGE0} and is independent of the Lam\'{e} coefficient $\lambda$. The mesh size $h$ is refined from $2^{-3}$ to $2^{-7}$ in 2D, and from $2^{-1}$ to $2^{-4}$ in 3D. The size parameter $\iota$ varies from $1$ to $10^{-6}$ in 2D and from $5\times10^{-2}$ to $5\times10^{-5}$ in 3D, while $\lambda$ takes $1$ and $10^6$.
\begin{example}\label{example1}
\normalfont
In two dimensions, we test the discrete method \eqref{IPDG} with the exact solution
$$
\boldsymbol{u}=
\left(
\begin{matrix}
\sin^3(\pi x_1)\sin(2\pi x_2)\sin(\pi x_2)\\
-\sin^3(\pi x_2)\sin(2\pi x_1)\sin(\pi x_1)
\end{matrix}
\right).
$$
Let
\begin{equation*}
{\rm Err}_{1}:=\|\boldsymbol{\varepsilon}_h(\boldsymbol{u}-\boldsymbol{u}_{h})\|_{0}+\sqrt{\lambda}\|\div(\boldsymbol{u}-\boldsymbol{u}_{h})\|_{0}.
\end{equation*}

Numerical errors $\interleave\boldsymbol{u}-\boldsymbol{u}_{h}\interleave_{\iota,\lambda,h}$ and $\rm {Err}_{1}$ are presented in Table~\ref{table12} and Table~\ref{table-sigma}, respectively.
We observe from Table \ref{table12} that ${\interleave\boldsymbol{u}-\boldsymbol{u}_{h}\interleave}_{\iota,\lambda,h}\eqsim \mathcal{O}(h)$ for $\iota = 1, 10^{-1}$, and ${\interleave\boldsymbol{u}-\boldsymbol{u}_{h}\interleave}_{\iota,\lambda,h}\eqsim \mathcal{O}(h^2)$ for $\iota = 10^{-5}, 10^{-6}$, which are optimal and consistent with the theoretical result \eqref{main-result-uh-1}. For $\iota \eqsim 1$, Table~\ref{table-sigma} indicates that ${\rm Err}_{1} \eqsim \mathcal{O}(h^2)$, in agreement with \eqref{H1-result}. Furthermore, a comparison of the results for $\lambda=1$ and $\lambda=10^6$ in both tables demonstrates that the convergence rates and error magnitudes are virtually unaffected by the large jump in $\lambda$. This confirms the $\lambda$-robustness of the proposed method.
\begin{table}
\centering
\caption{Error $\interleave\boldsymbol{u}-\boldsymbol{u}_{h}\interleave_{\iota,\lambda,h}$ of the discrete method \eqref{IPDG} for Example \ref{example1}.  }
\vspace{-1.0em}
\label{table12}
\begin{tabular}{ccccccc}
	\toprule
	$\lambda$ & $\iota\backslash h$ & 1/8 & 1/16 & 1/32 & 1/64 & 1/128\\
	\midrule
	\multirow{8}{*}{$1$}
	& 1 & 1.475e+01 & 8.359e+00 & 4.185e+00 & 2.053e+00 & 1.019e+00\\
	& rate &  & 0.82 & 1.00 & 1.03 & 1.01\\
	& $10^{-1}$ & 1.556e+00 & 8.382e-01 & 4.151e-01 & 2.046e-01 & 1.017e-01\\
	& rate &  & 0.89 & 1.01 & 1.02 & 1.01\\
	& $10^{-5}$ & 1.807e-01 & 4.776e-02 & 1.216e-02 & 3.054e-03 & 7.640e-04\\
	& rate &  & 1.92 & 1.97 & 1.99 & 2.00\\
	& $10^{-6}$ & 1.807e-01 & 4.776e-02 & 1.216e-02 & 3.054e-03 & 7.640e-04\\
	& rate &  & 1.92 & 1.97 & 1.99 & 2.00\\
	\midrule
	\multirow{8}{*}{$10^{6}$}
	& 1 & 1.518e+01 & 8.626e+00 & 4.305e+00 & 2.105e+00 & 1.050e+00\\
	& rate &  & 0.82 & 1.00 & 1.03 & 1.00\\
	& $10^{-1}$ & 1.603e+00 & 8.651e-01 & 4.269e-01 & 2.096e-01 & 1.040e-01\\
	& rate &  & 0.89 & 1.02 & 1.03 & 1.01\\
	& $10^{-5}$ & 1.822e-01 & 4.724e-02 & 1.189e-02 & 2.972e-03 & 7.425e-04\\
	& rate &  & 1.95 & 1.99 & 2.00 & 2.00\\
	& $10^{-6}$ & 1.822e-01 & 4.724e-02 & 1.189e-02 & 2.972e-03 & 7.424e-04\\
	& rate &  & 1.95 & 1.99 & 2.00 & 2.00\\
	\bottomrule
\end{tabular}
\vspace{1.0em}
\end{table}

\begin{table}
	\centering
	\caption{$\rm{Err}_1$ of the discrete method \eqref{IPDG} for Example \ref{example1}.}
	\vspace{-1.0em}
	\label{table-sigma}
	\begin{tabular}{ccccccc}
		\toprule
		$\lambda$ & $\iota\backslash h$ & 1/8 & 1/16 & 1/32 & 1/64 & 1/128\\
		\midrule
		\multirow{4}{*}{$1$}
		& $1$ & 8.354e-01 & 3.699e-01 & 1.208e-01 & 3.308e-02 & 8.486e-03\\
		& rate &  & 1.18 & 1.61 & 1.87 & 1.96\\
		& $10^{-1}$ & 5.760e-01 & 2.354e-01 & 7.346e-02 & 1.976e-02 & 5.044e-03\\
		& rate &  & 1.29 & 1.68 & 1.89 & 1.97\\
		\midrule
		\multirow{4}{*}{$10^{6}$}
		& $1$ & 7.539e-01 & 3.346e-01 & 1.108e-01 & 3.056e-02 & 7.842e-03\\
		& rate &  & 1.17 & 1.59 & 1.86 & 1.96\\
		&$10^{-1}$ & 5.182e-01 & 2.119e-01 & 6.691e-02 & 1.807e-02 & 4.605e-03\\
		& rate &  & 1.29 & 1.66 & 1.89 & 1.97\\
		\bottomrule
	\end{tabular}
	\vspace{1.0em}
\end{table}


\end{example}
\begin{example}\label{example2}
\normalfont
In three dimensions, we consider the exact solution
$$
\boldsymbol{u}=
\left(
\begin{matrix}
2\sin^3(\pi x)\sin(2\pi y)\sin(\pi y)\sin(2\pi z)\sin(\pi z)\\
-\sin^3(\pi y)\sin(2\pi z)\sin(\pi z)\sin(2\pi x)\sin(\pi x)\\
-\sin^3(\pi z)\sin(2\pi x)\sin(\pi x)\sin(2\pi y)\sin(\pi y)
\end{matrix}
\right).
$$
Table~\ref{table12-3d} shows that the convergence rate of $\interleave\boldsymbol{u}-\boldsymbol{u}_{h}\interleave_{\iota,\lambda,h}$ is $\mathcal{O}(h)$ for $\iota = 5\times10^{-2}$, and improves to $\mathcal{O}(h^2)$ for $\iota = 5\times10^{-4}, 5\times10^{-5}$. These results are optimal and align with the theoretical result \eqref{main-result-uh-1}. Similar to the 2D case, the results confirm the robustness of the method with respect to $\lambda$ in three dimensions.

\begin{table}
	\centering
	\caption{Error $\interleave\boldsymbol{u}-\boldsymbol{u}_{h}\interleave_{\iota,\lambda,h}$ of the discrete method \eqref{IPDG} for Example \ref{example2}.  }
	\vspace{-1.0em}
	\label{table12-3d}
	\begin{tabular}{cccccc}
		\toprule
		$\lambda$ & $\iota\backslash h$ & 1/2 & 1/4 & 1/8 & 1/16 \\
		\midrule
		\multirow{8}{*}{$1$}
		& $5\times10^{-2}$ & 3.440e+00 & 2.669e+00 & 1.724e+00 & 8.411e-01 \\
		& rate &  & 0.37 & 0.63 & 1.04 \\
		& $5\times10^{-3}$ & 1.941e+00 & 9.705e-01 & 3.481e-01 & 1.289e-01 \\
		& rate &  & 1.00 & 1.48 & 1.43 \\
		& $5\times10^{-4}$ & 1.756e+00 & 8.133e-01 & 2.502e-01 & 6.954e-02 \\
		& rate &  & 1.11 & 1.70 & 1.85 \\
		& $5\times10^{-5}$ & 1.735e+00 & 7.935e-01 & 2.383e-01 & 6.346e-02 \\
		& rate &  & 1.13 & 1.74 & 1.91 \\
		\midrule
		\multirow{8}{*}{$10^{6}$}
		& $5\times10^{-2}$ & 3.511e+00 & 2.738e+00 & 1.716e+00 & 8.705e-01 \\
		& rate &  & 0.36 & 0.67 & 0.98 \\
		& $5\times10^{-3}$ & 1.974e+00 & 1.007e+00 & 3.637e-01 & 1.343e-01 \\
		& rate &  & 0.97 & 1.47 & 1.44 \\
		& $5\times10^{-4}$ & 1.768e+00 & 8.300e-01 & 2.546e-01 & 7.067e-02 \\
		& rate &  & 1.09 & 1.70 & 1.85 \\
		& $5\times10^{-5}$ & 1.744e+00 & 8.082e-01 & 2.415e-01 & 6.399e-02 \\
		& rate &  & 1.11 & 1.74 & 1.92 \\
		\bottomrule
	\end{tabular}
	\vspace{1.0em}
\end{table}
\end{example}
\subsection{Numerical results with boundary layers}\label{sec52}
In this subsection, we evaluate the performance of the discrete method \eqref{IPDG} in the presence of boundary layers. We adopt divergence-free solutions of the linear elasticity problem \eqref{SGElinear}, which induce strong boundary layers as $\iota \rightarrow 0$. The right-hand side $\boldsymbol{f}$, computed from \eqref{SGElinear}, is used as the forcing term in \eqref{SGE0}, and notably, it is independent of both the size parameter $\iota$ and the Lam\'{e} coefficient $\lambda$. The tests use $\iota = 10^{-6}, 10^{-8}$ in 2D, and $\iota = 10^{-4}, 10^{-6}$ in 3D, with $\lambda = 1, 10^6$. The mesh refinement levels follow those in Examples \ref{example1} and \ref{example2}.
\begin{example}\label{example3}
\normalfont
In two dimensions, the exact solution of the linear elasticity model \eqref{SGElinear} is chosen as
$$
\boldsymbol{u}_0=
\left(
\begin{matrix}
(e^{\cos(2\pi x_1)}-e)\sin(2\pi x_2)e^{\cos(2\pi x_2)}\\
-(e^{\cos(2\pi x_2)}-e)\sin(2\pi x_1)e^{\cos(2\pi x_1)}
\end{matrix}
\right).
$$

Numerical error $\|\boldsymbol{u}_0-\boldsymbol{u}_{h}\|_{\iota,\lambda,h}$  of the discrete method \eqref{IPDG} for $\lambda = 1$ and $\lambda = 10^6$ is presented in Table~\ref{table34}. We can see from Table~\ref{table34} that $\|\boldsymbol{u}_0-\boldsymbol{u}_{h}\|_{\iota,\lambda,h}\eqsim \mathcal{O}(h^{2})$, which agrees with the estimate \eqref{main-result-uh-2}.
This confirms that the discrete method \eqref{IPDG} not only achieves optimal convergence but also maintains robustness with respect to the size parameter $\iota$ and the Lam\'{e} coefficient $\lambda$, even in the presence of strong boundary layers.

\begin{table}
	\centering
	\caption{Error $\|\boldsymbol{u}_0-\boldsymbol{u}_h\|_{\iota,\lambda,h}$ of the discrete method \eqref{IPDG} for Example \ref{example3}.}
	\vspace{-1.0em}
	\label{table34}
	\begin{tabular}{ccccccc}
		\toprule
		$\lambda$ & $\iota\backslash h$ & 1/8 & 1/16 & 1/32 & 1/64 & 1/128\\
		\midrule
		\multirow{4}{*}{$1$}
		& $10^{-6}$ & 1.289e+00 & 3.462e-01 & 8.784e-02 & 2.204e-02 & 5.513e-03\\
		& rate &  & 1.90 & 1.98 & 1.99 & 2.00\\
		& $10^{-8}$ & 1.289e+00 & 3.462e-01 & 8.784e-02 & 2.204e-02 & 5.513e-03\\
		& rate &  & 1.90 & 1.98 & 1.99 & 2.00\\
		\midrule
		\multirow{4}{*}{$10^{6}$}
		& $10^{-6}$ & 1.338e+00 & 3.492e-01 & 8.710e-02 & 2.164e-02 & 5.393e-03\\
		& rate &  & 1.94 & 2.00 & 2.01 & 2.00\\
		& $10^{-8}$ & 1.338e+00 & 3.492e-01 & 8.710e-02 & 2.164e-02 & 5.393e-03\\
		& rate &  & 1.94 & 2.00 & 2.01 & 2.00\\
		\bottomrule
	\end{tabular}
	\vspace{1.0em}
\end{table}


\end{example}
\begin{example}\label{example4}
\normalfont
In three dimensions, we employ the linear elasticity solution given by
$$
\boldsymbol{u}_0=
\left(
\begin{matrix}
2x^2(1-x)^2y(1-y)(1-2y)z(1-z)(1-2z)\\
-y^2(1-y)^2x(1-x)(1-2x)z(1-z)(1-2z)\\
-z^2(1-z)^2x(1-x)(1-2x)y(1-y)(1-2y)
\end{matrix}
\right).
$$
As shown in Table~\ref{table34-3d}, the error $\|\boldsymbol{u}_0-\boldsymbol{u}_{h}\|_{\iota,\lambda,h}\eqsim \mathcal{O}(h^2)$, which further confirms the robustness and optimality of the method \eqref{IPDG} with respect to both $\iota$ and $\lambda$.
\begin{table}
	\centering
	\caption{Error $\|\boldsymbol{u}_0-\boldsymbol{u}_h\|_{\iota,\lambda,h}$ of the discrete method \eqref{IPDG} for Example \ref{example4}.}
	\vspace{-1.0em}
	\label{table34-3d}
	\begin{tabular}{cccccc}
		\toprule
		$\lambda$ & $\iota\backslash h$ & 1/2 & 1/4 & 1/8 & 1/16 \\
		\midrule
		\multirow{4}{*}{$1$}
		& $10^{-4}$ & 1.837e-03 & 7.613e-04 & 2.174e-04 & 5.652e-05 \\
		& rate &  & 1.26 & 1.81 & 1.94 \\
		& $10^{-6}$ & 1.821e-03 & 7.565e-04 & 2.146e-04 & 5.496e-05 \\
		& rate &  & 1.27 & 1.82 & 1.96 \\
		\midrule
		\multirow{4}{*}{$10^{6}$}
		& $10^{-4}$ & 1.924e-03 & 8.024e-04 & 2.256e-04 & 5.814e-05 \\
		& rate &  & 1.26 & 1.82 & 1.96 \\
		& $10^{-6}$ & 1.916e-03 & 7.965e-04 & 2.232e-04 & 5.639e-05 \\
		& rate &  & 1.27 & 1.84 & 1.98 \\
		\bottomrule
	\end{tabular}
	\vspace{1.0em}
\end{table}
\end{example}

\subsection{Uniaxial tension of a plate with a circular hole}
\label{sec53}

We finally test the proposed method for the SGE model on a multiply connected domain. The reference problem is the uniaxial tension of an infinite plate containing a circular hole. Its domain is
\begin{equation}
	\Omega_\infty
	=\mathbb R^2\setminus\overline{B_R(\boldsymbol 0)},
	\qquad R=1~\mathrm{mm},
	\label{eq:perforated-infinite-domain}
\end{equation}
where \(B_R(\boldsymbol 0):=\{\boldsymbol x\in\mathbb R^2: \lvert\boldsymbol x\rvert<R\}\) is the open disk of radius \(R\) centered at the origin. No body force is applied, and the generalized traction and double traction vanish on the hole boundary \(\Gamma_{\rm h}:=\partial B_R(\boldsymbol 0)\).

The material is modeled under plane stress with Young's modulus \(E=70~\mathrm{GPa}\) and Poisson's ratio \(\nu=0.33\). The corresponding two-dimensional Lam\'{e} coefficients are
\[
	\mu=\frac{E}{2(1+\nu)}=26.316~\mathrm{GPa},
	\qquad
	\lambda=\frac{E\nu}{1-\nu^2}=25.923~\mathrm{GPa}.
\]
The plate is subjected at infinity to the uniform tensile stress \(\sigma_\infty=70~\mathrm{MPa}\) in the fixed Cartesian \(x\)-direction. At large distances from the hole, the Cauchy stress therefore approaches the uniform state with \(\sigma_{xx}=\sigma_\infty\) and \(\sigma_{yy}=\sigma_{xy}=0\). The corresponding remote axial strain is \(\varepsilon_0:=\sigma_\infty/E=10^{-3}\), a representative small strain.

For an unperforated plate, the corresponding displacement field is, up to a rigid motion,
\begin{equation}
	\boldsymbol u^\infty(x,y)
	=(\varepsilon_0x,-\nu\varepsilon_0y)^\intercal.
	\label{eq:perforated-far-field-displacement}
\end{equation}
For each $\iota>0$, an analytical displacement
$\boldsymbol u_\iota^{\rm KN}$ for the SGE model on the exterior
domain $\Omega_\infty$ defined in \eqref{eq:perforated-infinite-domain}
is given in \cite{khakalo2017gradient}.
In the notation of \cite[Secs.~3.1 and 4.2]{khakalo2017gradient}, the present uniaxial loading corresponds to \(\xi=0\). For \(\xi=0\), the corresponding displacement \(\boldsymbol u_\iota^{\rm KN}\) satisfies the natural conditions on \(\Gamma_{\rm h}\) and incorporates the perturbation induced by the hole and the size parameter. 
This perturbation decays with the distance from the hole, so the
Khakalo--Niiranen displacement approaches the far-field
displacement $\boldsymbol u^\infty$ given in
\eqref{eq:perforated-far-field-displacement}.

To state this solution, write \(\boldsymbol x=(r\cos\theta,r\sin\theta)^\intercal\) and set
\[
	\boldsymbol e_r=(\cos\theta,\sin\theta)^\intercal,\qquad
	\boldsymbol e_\theta=(-\sin\theta,\cos\theta)^\intercal.
\]
For any displacement field \(\boldsymbol v\), define
\begin{equation*}
	v_r:=\boldsymbol v\cdot\boldsymbol e_r,\qquad
	v_\theta:=\boldsymbol v\cdot\boldsymbol e_\theta,\qquad
	\sigma_{\theta\theta}(\boldsymbol v)
	:=\boldsymbol e_\theta^\intercal
	\boldsymbol\sigma(\boldsymbol v)\boldsymbol e_\theta.
\end{equation*}
Thus, \(\sigma_{\theta\theta}(\boldsymbol v)\) always denotes the Cauchy hoop stress generated by \(\boldsymbol v\).

In this basis, write
\[
\boldsymbol u_\iota^{\rm KN}
=u_{r,\iota}^{\rm KN}\boldsymbol e_r
+u_{\theta,\iota}^{\rm KN}\boldsymbol e_\theta
\]
and take its radial functions from \cite[Eqs.~(35), (39) and Appendix~A]{khakalo2017gradient}:
\begin{equation*}
	u_{r,\iota}^{\rm KN}(r,\theta)
	=\widetilde u_r(r)+\widehat u_r(r)\cos(2\theta),
	\qquad
	u_{\theta,\iota}^{\rm KN}(r,\theta)
	=\widehat u_\theta(r)\sin(2\theta),
\end{equation*}
where \(\gamma=\mu/\lambda\), \(z=r/\iota\), and
\begin{subequations}\label{eq:perforated-khakalo-radial}
\begin{align}
	\widetilde u_r(r)
	&=A_1r+\frac{A_2}{r}+A_4K_1(z),\\
	\widehat u_r(r)
	&=-C_1r-\frac{1+2\gamma}{\gamma}\frac{C_2}{r}
	+\frac{C_3}{r^3}-C_7K_1(z) \notag\\
	&\quad
	+C_8\left[
	\left(1+\frac{8\iota^2}{r^2}\right)K_1(z)
	+\frac{4\iota}{r}K_0(z)
	\right],\\
	\widehat u_\theta(r)
	&=C_1r+\frac{C_2}{r}+\frac{C_3}{r^3}
	+C_7K_1(z)+C_8K_3(z).
\end{align}
\end{subequations}
Here, \(K_n\) denotes the modified Bessel function of the second kind of order \(n\). The coefficients \(A_1=\varepsilon_0(1-\nu)/2\) and \(C_1=-\varepsilon_0(1+\nu)/2\) make the \(r\)-growing terms reproduce the remote deformation \(\boldsymbol u^\infty\). In terms of the coefficients used in \cite[Eq.~(39)]{khakalo2017gradient}, the two inverse-power coefficients in \eqref{eq:perforated-khakalo-radial} are
\[
C_2=R(C_7\alpha_1+C_8\alpha_2),
\qquad
C_3=R^3(C_7\alpha_3+C_8\alpha_4).
\]
The dimensionless quantities \(\alpha_1,\alpha_2,\alpha_3,\alpha_4\) and the coefficients \(A_2,A_4,C_7,C_8\) are given explicitly in Appendix~A of \cite{khakalo2017gradient}.
With these definitions, the collected form \eqref{eq:perforated-khakalo-radial} is algebraically identical to \cite[Eq.~(39)]{khakalo2017gradient}.

To approximate the infinite-plate problem on a finite mesh, we truncate \(\Omega_\infty\) to the computational domain shown in Figure~\ref{fig:perforated-domain}:
\begin{equation}
	\Omega_L=(-L,L)^2\setminus\overline{B_R(\boldsymbol 0)},
	\qquad L=8R.
	\label{eq:perforated-domain}
\end{equation}
On the artificial outer boundary \(\Gamma_{\rm o}:=\partial(-L,L)^2\), we prescribe only its displacement trace,
\begin{equation}
	\boldsymbol u=\boldsymbol u_\iota^{\rm KN}
	\qquad\text{on }\Gamma_{\rm o}.
	\label{eq:perforated-outer-data}
\end{equation}
Equation~\eqref{eq:perforated-outer-data} is the only essential condition on \(\Gamma_{\rm o}\); all remaining boundary terms there are assigned their homogeneous natural values.
Following \cite[Remark~4.9]{HuangHuangTang2024}, we adapt the discrete formulation to the present mixed boundary conditions. The prescribed displacement on $\Gamma_{\rm o}$ is imposed through the corresponding boundary DoFs, whereas the generalized-traction and double-traction data on $\Gamma_{\rm h}$, together with the remaining natural data on $\Gamma_{\rm o}$, are incorporated through the boundary functionals on the right-hand side. Since all the natural data considered here are homogeneous, the corresponding boundary load terms vanish.

Thus, \eqref{eq:perforated-outer-data} uses the exact displacement trace to approximate the omitted exterior part of the infinite plate. 
Since the remaining outer-boundary data are not taken from the
infinite-domain solution, its restriction to the truncated domain
$\Omega_L$ defined in \eqref{eq:perforated-domain} is generally not
the exact solution of the corresponding truncated problem. The reported differences from \(\boldsymbol u_\iota^{\rm KN}\) therefore contain an outer-boundary truncation effect in addition to the discretization and polygonal-geometry errors.

The circular boundary is approximated by a regular \(64\)-gon with apothem \(R\). The mesh is fitted to this polygon and is nonuniform in its neighborhood. Away from the hole, a uniform triangulation of right isosceles triangles with leg length \(h_{\rm b}\) is used. We use \(h_{\rm b}/R = 1,1/2,1/4,1/8\) and set the penalty parameter \(\eta=20\). The central part of the finest mesh is shown in Figure~\ref{fig:perforated-mesh}. To examine the size effects within the SGE regime, we test \(\iota/R = 0.5,0.2,0.1,0.05\).

\begin{figure}[htbp]
	\centering
	\captionsetup[subfigure]{skip=2pt}
	\setlength{\tabcolsep}{2pt}

	\begin{tabular}{cc}

		\begin{subfigure}{0.4\textwidth}
			\centering
			\includegraphics[width=\linewidth]{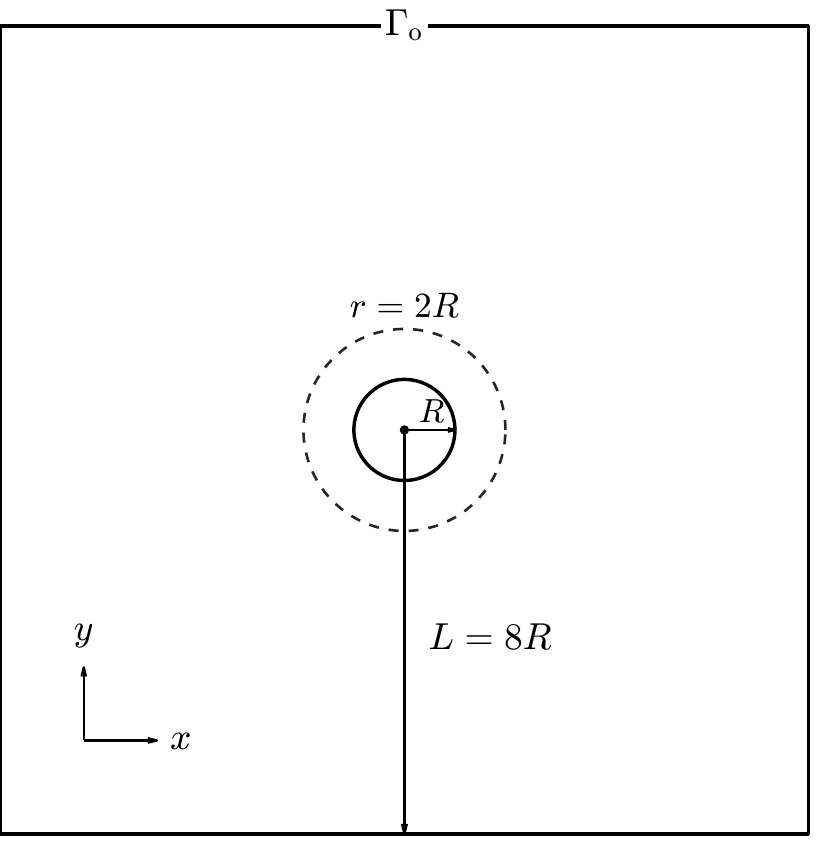}
			\caption{Computational domain.}
			\label{fig:perforated-domain}
		\end{subfigure}
		\hspace{0.04\textwidth}
		\begin{subfigure}{0.4\textwidth}
			\centering
			\includegraphics[width=\linewidth]{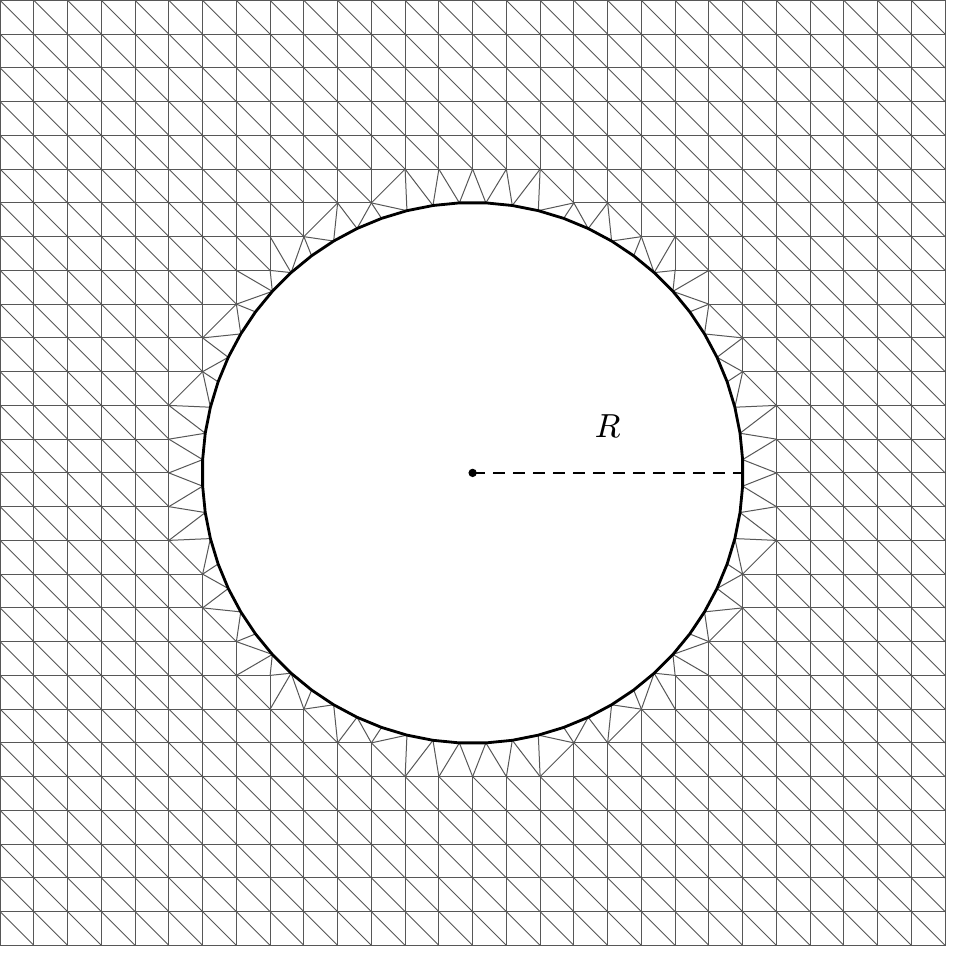}
			\caption{Local mesh for \(h_{\rm b}/R=1/8\).}
			\label{fig:perforated-mesh}
		\end{subfigure}
	\end{tabular}
	\caption{Truncated circular-hole domain and a local view of the computational mesh.}
	\label{fig:perforated-model-mesh}
\end{figure}

The comparison is made on the interior circle \(r=2R\), see Figure~\ref{fig:perforated-domain}. This sampling circle remains close to the stress concentration while avoiding direct stress evaluation at the corners of the polygonal hole. For each fixed \(\iota>0\), let \(\boldsymbol u_h\) denote the corresponding discrete solution, and define
\[
	u_{r,h}:=\boldsymbol u_h\cdot\boldsymbol e_r,
	\qquad
	\sigma_{\theta\theta,h}(\boldsymbol u_h)
	:=\boldsymbol e_\theta^\intercal
	\boldsymbol\sigma_h(\boldsymbol u_h)\boldsymbol e_\theta.
\]
For \(q\in\{\sigma_{\theta\theta},u_r\}\), set
\[
	(q_h,q_\iota^{\rm KN})
	=
	\begin{cases}
	\big(\sigma_{\theta\theta,h}(\boldsymbol u_h),
	\sigma_{\theta\theta}(\boldsymbol u_\iota^{\rm KN})\big),
	& q=\sigma_{\theta\theta},\\
	\big(u_{r,h},
	u_{r,\iota}^{\rm KN}\big),
	& q=u_r.
	\end{cases}
\]
The relative errors are evaluated at $N_\theta=144$ equally spaced points on the sampling circle $r=2R$. Let $\theta_j$, $j=1,\ldots,N_\theta$, denote their angular coordinates. A fixed rotation is applied to the angular grid so that, for all meshes considered, none of the sampling points lies on the mesh skeleton. Consequently, the nonconforming displacement and the elementwise Cauchy stress are evaluated unambiguously from a single element.
The relative discrete error is defined by
\begin{equation*}
	e_h(q;\iota)=
	\frac{\left(\displaystyle\sum_{j=1}^{N_\theta}
		|q_h(2R,\theta_j)-q_\iota^{\rm KN}
		(2R,\theta_j)|^2\right)^{1/2}}
	{\left(\displaystyle\sum_{j=1}^{N_\theta}
		|q_\iota^{\rm KN}(2R,\theta_j)|^2\right)^{1/2}}.
\end{equation*}

As shown in Table~\ref{tab:perforated-khakalo-multiscale}, for each tested value of \(\iota/R\), the errors in both the Cauchy hoop stress and the radial displacement decrease as \(h_{\rm b}/R\) is reduced. This agreement with the Khakalo--Niiranen solution demonstrates the applicability of the proposed method to a geometrically nontrivial perforated domain.

\begin{table}[H]
	\centering
	\caption{Relative discrete errors on the sampling circle \(r=2R\).}
	\vspace{-1.0em}
	\label{tab:perforated-khakalo-multiscale}
	\begin{tabular}{cccccc}
		\toprule
		\(\iota/R\) & \(e_h(\cdot;\iota)\backslash(h_{\rm b}/R)\) & \(1\) & \(1/2\)
		& \(1/4\) & \(1/8\)\\
		\midrule
		\multirow{2}{*}{\(0.5\)}
		& \(e_h(\sigma_{\theta\theta};\iota)\)
		& \(2.0443\%\) & \(0.8962\%\) & \(0.3262\%\) & \(0.1181\%\)\\
		& \(e_h(u_r;\iota)\)
		& \(3.1247\%\) & \(1.7965\%\) & \(0.6870\%\) & \(0.2353\%\)\\
		\midrule
		\multirow{2}{*}{\(0.2\)}
		& \(e_h(\sigma_{\theta\theta};\iota)\)
		& \(4.4815\%\) & \(1.5611\%\) & \(0.4424\%\) & \(0.1637\%\)\\
		& \(e_h(u_r;\iota)\)
		& \(2.5112\%\) & \(1.4418\%\) & \(0.4020\%\) & \(0.1351\%\)\\
		\midrule
		\multirow{2}{*}{\(0.1\)}
		& \(e_h(\sigma_{\theta\theta};\iota)\)
		& \(3.7448\%\) & \(2.6589\%\) & \(0.3465\%\) & \(0.1004\%\)\\
		& \(e_h(u_r;\iota)\)
		& \(2.1140\%\) & \(1.4933\%\) & \(0.1832\%\) & \(0.0555\%\)\\
		\midrule
		\multirow{2}{*}{\(0.05\)}
		& \(e_h(\sigma_{\theta\theta};\iota)\)
		& \(6.8509\%\) & \(1.5594\%\) & \(0.2721\%\) & \(0.0618\%\)\\
		& \(e_h(u_r;\iota)\)
		& \(1.1529\%\) & \(0.5551\%\) & \(0.0778\%\) & \(0.0152\%\)\\
		\bottomrule
	\end{tabular}
	\vspace{1.0em}
\end{table}

%

Fixing \(\theta=\pi/2\), we compare the normalized Cauchy hoop-stress profiles as functions of \(r\).  For the numerical and analytical solutions, the plotted quantities are, respectively, \(\{\sigma_{\theta\theta,h}(\boldsymbol u_h)\}/\sigma_\infty\) and \(\sigma_{\theta\theta} (\boldsymbol u_\iota^{\rm KN})/\sigma_\infty\). Both profiles are evaluated at \(180\) equally spaced radial positions over \(1.13\leq r/R\leq2.99\). Since the sampling ray coincides with interior mesh edges over this interval, the numerical stress at each sampling point is obtained by averaging the values from the two adjacent elements. The resulting radial profiles are shown in Figure~\ref{fig:perforated-khakalo-radial-stress}.

\begin{figure}[htbp]
	\centering
	\begin{subfigure}[t]{0.48\textwidth}
		\centering
		\includegraphics[width=\textwidth]
		{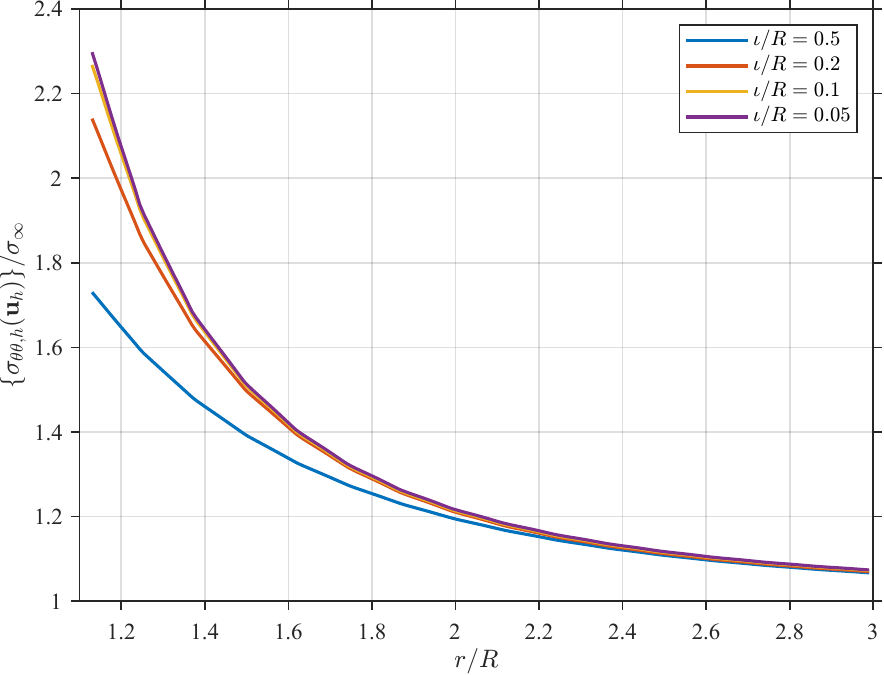}
		\caption{Numerical profiles.}
	\end{subfigure}
	\hfill
	\begin{subfigure}[t]{0.48\textwidth}
		\centering
		\includegraphics[width=\textwidth]
		{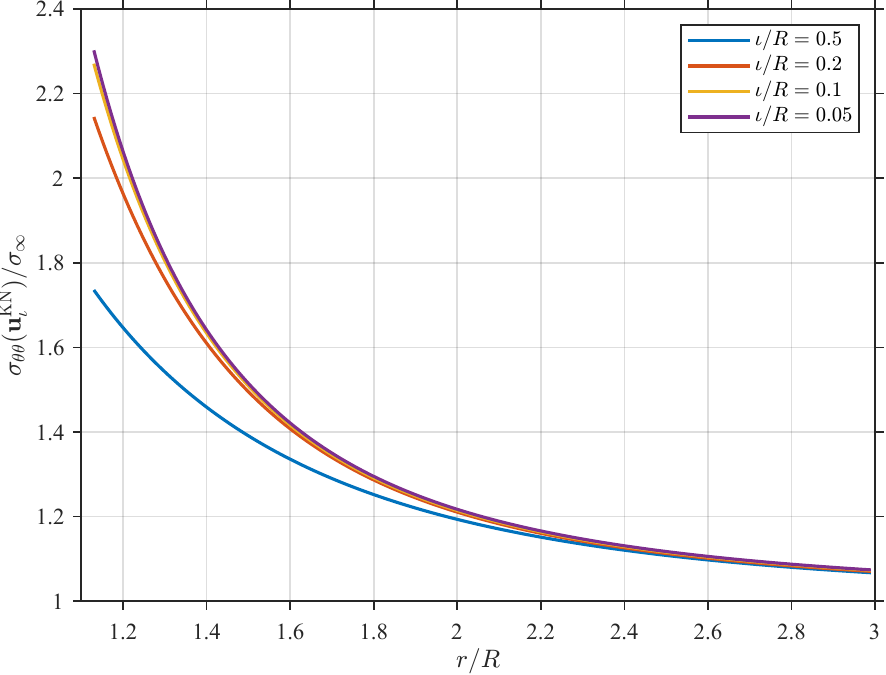}
		\caption{Khakalo--Niiranen profiles.}
	\end{subfigure}
	\caption{Normalized Cauchy hoop-stress profiles along
		\(\theta=\pi/2\) for \(h_{\rm b}/R=1/8\) and
		\(1.13\leq r/R\leq2.99\).}
	\label{fig:perforated-khakalo-radial-stress}
\end{figure}

Both the numerical and Khakalo--Niiranen profiles show that increasing \(\iota/R\) lowers the normalized Cauchy hoop stress in the vicinity of the hole. Moreover, the differences among the four values of \(\iota/R\) decay rapidly with the distance from the hole, and all profiles approach the normalized far-field value \(1\). Thus, the size parameter primarily affects the near-hole stress field, in agreement with the analytical observations in \cite[Fig.~4(a)]{khakalo2017gradient}.

\bibliographystyle{abbrv} 
\bibliography{ref} 
\end{document}